\documentclass[reqno,12pt,  letterpaper]{amsart}
\usepackage{amsmath,amssymb,amsthm,graphicx,mathrsfs,url}
\usepackage[usenames,dvipsnames]{color}
\usepackage[colorlinks=true,linkcolor=Red,citecolor=Green]{hyperref}

\usepackage{mathtools}
\DeclarePairedDelimiter\ceil{\lceil}{\rceil}

\def\?[#1]{\textbf{[#1]}\marginpar{\Large{\textbf{??}}}}

\newtheorem{thm}{Theorem}
\newtheorem{prop}{Proposition}[section]

\newtheorem{lem}[prop]{Lemma}
\newtheorem{cor}[prop]{Corollary}
\newtheorem{rem}[prop]{Remark}
\numberwithin{equation}{section}

\DeclareMathOperator{\supp}{supp}

\newcommand{\mc}{\mathcal}
\newcommand{\rr}{\mathbb{R}}

\newcommand{\cc}{\mathbb{C}}

\newcommand{\la}{\lambda}
\newcommand{\eps}{\varepsilon}

\newcommand{\pl}{\partial}
\newcommand{\x}{\times}

\newcommand{\bbar}{\overline}

\newcommand{\cjd}{\rangle}
\newcommand{\cjg}{\langle}

\newcommand{\ndemi}{\frac{n}{2}}

\newcommand\jap[1]{\langle #1 \rangle}

\renewcommand\Re{\mathrm{Re}\,  }

\newcommand{\R}{\mathbb{R}}

\newcommand{\C}{\mathbb{C}}
\newcommand{\N}{\mathbb{N}}

\newcommand{\RR}{\mathbb{R}}

\newcommand{\CC}{\mathbb{C}}
\newcommand{\NN}{\mathbb{N}}
\newcommand\Id{\operatorname{Id}}
\newcommand\sskkip[1]{}

\renewcommand{\Re}{\hbox{Re}\,}
\renewcommand{\Im}{\hbox{Im}\,}

\newcommand{\p}{\partial}

\newcommand\MMkb{M^2_{k, b}}

\newcommand{\SC}{\ensuremath{\mathrm{sc}}}

\newcommand\Lsharp{L^\#}

\newcommand{\zf}{\mathrm{zf}}
\newcommand{\bfo}{\mathrm{bf}_0}
\newcommand{\rbo}{\mathrm{rb}_0}
\newcommand{\lbo}{\mathrm{lb}_0}
\newcommand{\lb}{\mathrm{lb}}
\newcommand{\rb}{\mathrm{rb}}

\newcommand{\bfc}{\mathrm{bf}}

\newcommand\mcB{\mathcal{B}}

\newcommand\Omegakbh{\Omega_{k,b}^{1/2}}

\newcommand{\Nscstar}[1][\mbox{}]{{^{\SC}N^*_{#1}}}
\newcommand\Diagb{\text{Diag}_b}

\title[Eigenvalue bounds for non-self-adjoint Schr\"odinger operators]{Eigenvalue bounds for non-self-adjoint Schr\"odinger operators with non-trapping metrics}

\author{Colin Guillarmou}
\address{Laboratoire de Math\'ematiques d'Orsay\\ 
Universit\'e Paris-Sud\\  
 Bat. 425, 91405 Orsay Cedex \\
France}
\email{colin.guillarmou@math.u-psud.fr}

\author{Andrew Hassell}
\address{Department of Mathematics\\
Australian National University\\
Canberra ACT 0200 Australia}
\email{hassell@maths.anu.edu.au}

\author{Katya Krupchyk}
\address{Department of Mathematics\\
University of California, Irvine\\ 
CA 92697-3875, USA }

\email{katya.krupchyk@uci.edu}

\begin{document}

\begin{abstract}

We study eigenvalues of non-self-adjoint Schr\"odinger operators on non-trapping asymptotically conic manifolds of dimension $n\ge 3$.  Specifically, we are concerned with the following two types of estimates. The first one deals with Keller type bounds on individual eigenvalues of the Schr\"odinger operator with a complex potential in terms of the $L^p$-norm of the potential, while the second one is a  Lieb--Thirring type bound controlling sums of powers of eigenvalues in terms of the $L^p$-norm of the potential. We extend the results of Frank (2011), Frank--Sabin (2017), and Frank--Simon (2017) on the Keller and Lieb--Thirring type bounds from the case of Euclidean spaces to that of non-trapping asymptotically conic manifolds. In particular, our results are valid for 
the operator $\Delta_g+V$ on $\rr^n$ with $g$ being a non-trapping compactly supported (or suitably short range) perturbation of the  Euclidean metric and $V\in L^p$ complex valued. 

\end{abstract}

\maketitle

\section{Introduction and statement of results}
The purpose of this paper is to establish bounds of Keller and Lieb--Thirring type for eigenvalues of non-self-adjoint Schr\"odinger operators on non-trapping asymptotically conic manifolds. Before stating our results, let us proceed to describe these two types of bounds in the more familiar Euclidean setting, motivating the significance of extending them to the case of asymptotically conic manifolds.  

\subsection{Keller and Lieb--Thirring type bounds in the Euclidean case}
Recently there have been numerous works devoted to the study of eigenvalues of  the Schr\"odinger operator  $\mathcal{P}=\Delta+V$ in $L^2(\R^n)$, with $\Delta$ being the nonnegative Laplace operator and $V$ being  a complex-valued potential.  Of particular interest here is the problem of obtaining quantitative information concerning the localization and distribution of the eigenvalues of $\mathcal{P}$ under the sole assumption that $V\in L^p(\R^n)$,  for some $1 \le p<\infty$.  Here we may remark that the spectrum of $\mathcal{P}$ in $\C\setminus [0,\infty)$ consists then of isolated eigenvalues of finite algebraic multiplicity, see \cite[Proposition B.2]{Frank_2015}. 

The following two types of results are of particular interest for this problem.   The first one deals with Keller type bounds on the individual eigenvalues of $\mathcal{P}$ in terms of the $L^ p$-norm of the potential, \cite{Keller_61}. If $V$ is real-valued, so that $\mathcal{P}$ admits a natural self-adjoint realization, then the eigenvalues of $\mathcal{P}$ in $\C\setminus [0,\infty)$ are negative and by the variational principle and Sobolev's inequalities, for any eigenvalue $\lambda<0$ of $\mathcal{P}$, we have the scale-invariant bounds, 
\begin{equation}
\label{eq_intro_1}
|\lambda|^{\gamma}\le C_{\gamma,n}\int_{\R^n} |V(x)|^{\gamma+\frac{n}{2}}dx
\end{equation}
for every $\gamma\ge\frac{1}{2}$ if $n=1$ and every $\gamma>0$ if $n\ge 2$. Here the constant $C_{\gamma,n}>0$ depends on $\gamma$ and $n$ only, see \cite{Frank_Simon}, \cite{Keller_61}, \cite{Lieb_Thirring}.   

If the potential $V$ is complex-valued, the problem is more involved due to the lack of variational techniques and the absence of a spectral resolution theorem. In dimension  $n=1$ the bound \eqref{eq_intro_1}  with  $\gamma=\frac{1}{2}$ was proved  by Abramov, Aslanyan, and Davies in \cite{Abramov_Aslanyan_Davies}. In dimensions  $n\ge 2$,  Frank \cite{Frank_2011} established  the bound \eqref{eq_intro_1} for  all eigenvalues $\lambda\in \C\setminus [0,\infty)$ and for all  $0<\gamma\le \frac{1}{2}$, see also \cite{Frank_Simon}. The work \cite{Frank_2015} gives a replacement of the bound  \eqref{eq_intro_1}  for all $\gamma>\frac{1}{2}$. 
We refer to  \cite{Cuenin}, \cite{Cuenin_Kenig}, \cite{Enblom_A},  \cite{Laptev_Safronov}, \cite{Mizutani_critical}, for some other recent works on bounds on the individual eigenvalues for non-self-adjoint operators of Schr\"odinger type.

The second type of result is concerned with bounds on sums of powers of absolute values of eigenvalues of $\mathcal{P}$, generalizing the 
classical Lieb--Thirring bounds \cite{Lieb_Thirring}  to the non-self-adjoint case. If $V$ is real-valued then the Lieb--Thirring  inequality has the following form,
\begin{equation}
\label{eq_intro_2}
\sum |\lambda|^{\gamma}\le C_{\gamma,n}\int_{\R^n}V_-(x)^{\gamma+\frac{n}{2}}dx,
\end{equation}
where $V_-=\max(-V,0)$,  $\gamma\ge \frac{1}{2}$ if $n=1$, $\gamma>0$ if $n=2$, and $\gamma\ge 0$ if $n\ge 3$.   The summation in the left hand side in \eqref{eq_intro_2} extends over all negative eigenvalues of $\mathcal{P}$, counted with their multiplicities.  The situation in the non-self-adjoint case is less clear. In particular, B\"ogli \cite{Bogli_Sabine}  established that for any $p>n$, there exists  a non-real potential $V\in L^p(\R^n)\cap L^\infty(\R^n)$ such that the Schr\"odinger operator $\mathcal{P}$ has infinitely many non-real eigenvalues accumulating at every point of the essential spectrum $[0,\infty)$, thus showing that inequalities like \eqref{eq_intro_2} cannot hold in the non-self-adjoint case for $p>n$.  A possible modification of Lieb--Thirring's inequality \eqref{eq_intro_2} to the non-self-adjoint case was suggested in  \cite{Demuth_Hansmann_Katriel_LT}, and is as follows, 
\begin{equation}
\label{eq_intro_2_new}
\sum \frac{d(\lambda)^{\gamma+\frac{n}{2}}}{|\lambda|^{\frac{n}{2}}}\le C_{\gamma,n}\int_{\R^n}|V(x)|^{\gamma+\frac{n}{2}}dx, 
\end{equation}
where
\begin{equation}
\label{eq_intro_2_new_d}
d(\lambda)=\text{dist}(\lambda, [0,\infty)).
\end{equation}
We refer to \cite{Demuth_Hansmann_Katriel_2009},  \cite{Demuth_Hansmann_Katriel_2app}, \cite{Frank_Laptev_Lieb_Seiringer}, \cite{Frank_Sabin}, \cite{Sambou}  for some of the important contributions to generalizations of Lieb--Thirring's inequality \eqref{eq_intro_2}  to the setting of complex potentials.

A crucial idea of Frank \cite{Frank_2011} in establishing bounds \eqref{eq_intro_1} on the individual eigenvalues of the Schr\"odinger operator  $\mathcal{P}$ with a complex-valued potential was to make use of the uniform $L^p$ resolvent estimates for $\Delta$ of Kenig, Ruiz, Sogge \cite{Kenig_Ruiz_Sogge}. Recently, this approach was extended to the case of  non-self-adjoint Schr\"odinger operators with  inverse-square potentials by Mizutani \cite{Mizutani}, to the case of magnetic Schr\"odinger and Pauli operators with complex electromagnetic potentials by Cuenin and Kenig  \cite{Cuenin_Kenig}, and to the case of the Dirac and fractional Schr\"odinger operators with complex potentials by Cuenin \cite{Cuenin}.

Developing the idea of Frank \cite{Frank_2011}  further, Frank and Sabin  \cite{Frank_Sabin} obtained some very interesting uniform weighted bounds for the resolvent of $\Delta$ in suitable Schatten classes, and applied these bounds to derive uniform estimates on the sums of eigenvalues of non-self-adjoint Schr\"odinger operators, thus obtaining some results towards proving  the conjectured Lieb--Thirring inequality \eqref{eq_intro_2_new} in the case of complex potentials.  Recently, this approach was extended  by Cuenin \cite{Cuenin} to the case of the Dirac and fractional Schr\"odinger operators with complex potentials. 

\subsection{Asymptotically conic manifolds}
Notice that in all the works described above the principal part of the operators considered has constant coefficients. It is nevertheless of significant interest to extend both types of results to the case of complex potential perturbations of the  Laplace--Beltrami operator $\Delta_g$ considered on $\R^n$ or more generally, on a  class of complete non-compact Riemannian manifolds.   

The class we consider here is the class of \emph{asymptotically conic manifolds}, whose Riemannian metric outside a compact set is asymptotic to the end of a metric cone. Metric cones are Riemannian manifolds of the form $N \times (0, \infty)_r$ with metric $dr^2 + r^2 G$, for some metric $G$ on $N$. They were studied by Cheeger \cite{Ch} and Cheeger-Taylor \cite{CT} but have a long history going back to Sommerfeld \cite{Sommerfeld}. As defined by Melrose \cite{Melrose} (who used the term `scattering metric'), 
$(M,g)$ is \emph{asymptotically conic} if $M$ is the interior of a smooth compact manifold with boundary $\overline{M}$ and $g$ is a smooth metric on $M$ satisfying the following property: 
there exists a smooth boundary defining function\footnote{A boundary defining function is a non-negative smooth function $x$ such that $\pl\bbar{M}=x^{-1}(\{0\})$ and $dx|_{\pl M}$ does not vanish on $\pl \bbar{M}$.} $x$ on $\overline{M}$ such that $(M,g)$ is isometric outside a compact set to a collar $(0,\epsilon)_x\times \p \overline{M}$ equipped with the metric of the form  
\begin{equation}
\label{asym_con_metric_intr_1}
\frac{dx^2}{x^4}+\frac{h(x)}{x^2}=\frac{dx^2}{x^4}+\frac{\sum_{j,k}h_{jk}(x,y)dy_jdy_k}{x^2}
\end{equation}
for some smooth one-parameter family of metrics $h$ on the boundary $\p \overline{M}$.  If $y=(y_1,\dots, y_{n-1})$ stand for local coordinates on $\p  \overline{M}$ and $(x,y)$ are the corresponding local coordinates on $M$ near $\p \overline{M}$, the function $r=1/x$ near $x=0$ can be thought of as a ``radial" variable near infinity and $y=(y_1,\dots, y_{n-1})$ can be regarded as $n-1$ ``angular" variables. Rewriting \eqref{asym_con_metric_intr_1} in the $(r,y)$ coordinates,  
\begin{equation}
\label{asym_con_metric_intr}
g=dr^2+ r^2 h(r^{-1})= dr^2+ r^2 \sum h_{jk}(r^{-1},y)dy^jdy^k,
\end{equation}
and we observe that the metric $g$ is asymptotic to an exact conic metric $dr^2+r^2h(0)$ on $(r_0,\infty)_r\times \p \overline{M}$ as $r\to \infty$.
The most important example of an asymptotically conic manifold is Euclidean space $M=\R^n$
equipped with a short range perturbation of the Euclidean metric $(\delta_{ij})$, which is of the form 
\begin{equation}
\label{eq_short_range}
g_{ij}=\delta_{ij}+|z|^{-2} k_{ij}\bigg(\frac{z}{|z|},\frac{1}{|z|}\bigg), \quad |z|\to \infty,
\end{equation}
where $k_{ij}$ are smooth on $\mathbb{S}^{n-1}\times [0,1)$ --- see \cite{Melrose_Zworski}.

 Let  $z=(z_1,\dots, z_n)$ be local coordinates away from $\p \overline{M}$. We say that $M$ is \emph{non-trapping} if every geodesic $z(s)$ in $M$ reaches $\p \overline{M}$ as $s\to \pm \infty$.  This places restrictions on the compactification $\bbar{M}$. 
For example, a compact perturbation of the Euclidean metric is non-trapping provided that it is sufficiently small in $C^2$ --- see \cite{HTW}. However, a non-trapping asymptotically conic metric $g$ may be far from asymptotically Euclidean. Indeed, there is such a non-trapping metric $g$ on $\mathbb{R}^n$ for every limiting metric $h(0)$ on the sphere $\mathbb{S}^{n-1}$, identified with $\pl \bbar{M}$ in this case.

In terms of the Weyl calculus, the symbol of the Laplacian for an asymptotically conic metric on $\R^n$ is in the calculus corresponding to the metric on $T^*\R^n$
$$
\frac{dz^2}{\jap{z}^2} + \frac{d\xi^2}{\jap{\xi}^2}.
$$
This class of symbols was studied by Parenti \cite{Parenti}, Cordes \cite{Cordes}, Schrohe \cite{Schrohe}, H\"ormander \cite[Equation (19.3.11) and Theorem 19.3.1$'$]{Ho3} and others. Melrose \cite{Melrose} adopted a different point of view, working from the outset on the compactification $\bbar{M}$ (which can be any manifold with boundary) and introducing the  \emph{scattering calculus} as the natural class of pseudodifferential operators associated with the \emph{scattering Lie algebra} of vector fields on $\bbar{M}$. He seems to have been the first to  exploit the fact that in this calculus one has propagation of singularities at spatial infinity \emph{at all finite frequencies}. Using the scattering calculus, the second author in collaboration with Vasy, Wunsch, the first author, and Sikora, worked out detailed properties of the spectral measure --- see \cite{GuHaSi_III, HaVa, HaWu2008}.

Let us remark on why we elect to work with the class of non-trapping asymptotically conic manifolds. 
On the one hand, it is a sufficiently \emph{general} class which  includes compactly supported or suitable short-range perturbations of Euclidean space as well as geometrically interesting examples such as metrics with strictly negative curvature, which are not present in the class of asymptotically Euclidean manifolds. On the other hand, it is sufficiently \emph{restricted} to allow to obtain detailed results on the resolvent and spectral measure, analogous in some sense to that for flat Euclidean space.

\subsection{Main results}

Throughout the paper, we let $(M,g)$ be an asymptotically conic non-trapping manifold of dimension $n\ge 3$.    
Since $g$ is complete, the Laplacian $\Delta_g$ associated with the metric $g$ is nonnegative self-adjoint on $L^2(M)$ with domain $H^2(M)$. The spectrum of $\Delta_g$ is purely absolutely continuous and is given by $\text{Spec}(\Delta_g)=[0,\infty)$: the absence of singular continuous spectrum follows for example from \cite{FrHi} using a Mourre estimate, and the absence of embedded $L^2$-eigenvalues follows from adapting \cite[Theorem 17.2.8]{Ho3} as in  \cite[Section 10]{Melrose}.

Our starting point is the following uniform $L^p$ resolvent estimates of the Kenig--Ruiz--Sogge type for the Laplace operator $\Delta_g$ on an asymptotically  conic non-trapping manifold, established in the work \cite{GuHa} of the first two authors. 

\begin{thm}
\label{thm_GuHa}
Let $(M,g)$ be an asymptotically conic non-trapping manifold of dimension $n\ge 3$.  Then for all $p\in [\frac{2n}{n+2}, \frac{2(n+1)}{n+3}]$, there is a constant $C>0$ such that for all $z\in \C$ and for all $f\in L^p(M)$, we have
\begin{equation}
\label{eq_KRS_manifold}
\| (\Delta_g -z)^{-1}f\|_{L^{p'}(M)}\le C|z|^{n(\frac{1}{p}-\frac{1}{2})-1}\|f\|_{L^p(M)}.
\end{equation}
Here  $\frac{1}{p}+\frac{1}{p'}=1$. 
\end{thm}

As explained in \cite{GuHa}, when $z\in (0,+\infty)$, the operator in \eqref{eq_KRS_manifold} may be taken to be either the outgoing or incoming resolvent $(\Delta_g-(z\pm i0))^{-1}$, defined by 
\[
(\Delta_g-(z\pm i0))^{-1}=\lim_{\delta\to 0^+}(\Delta_g-(z\pm i\delta))^{-1}
\] 
as a map $x^{1/2+\eps}L^2(M)\to x^{-1/2-\eps}L^2(M)$ for all 
$\eps>0$, where $x$ is the boundary defining function, thanks to   the limiting absorption principle,   see \cite{HaVa}, \cite{Melrose} for details.

The main technical contribution of the present paper is the following weighted uniform Schatten class estimate for the resolvent of $\Delta_g$, generalizing a result of Frank and Sabin \cite[Theorem 12]{Frank_Sabin}, obtained in the Euclidean setting.  This result is the key ingredient which allows us to extend the Lieb-Thirring type bounds of Frank-Sabin \cite{Frank_Sabin} and Frank \cite{Frank_2015} to our setting.
Below, $\mc{C}_q(L^2(M))$ denotes the Schatten space of order $q$ (see Section \ref{schattennormest} for definition).
\begin{thm}
\label{thm_resolvent_Schatten_laplacian}
Let $(M,g)$ be an asymptotically conic non-trapping manifold of dimension $n\ge 3$.   Let $p\in [\frac{n}{2}, \frac{n+1}{2}]$. 
Then there exists $C>0$  such that  for all $z\in \C\setminus \{0\}$ and all $W_1,W_2 \in L^{2p}(M)$, we have $W_1(\Delta_g-z)^{-1} W_2\in \mathcal{C}_{q}(L^2(M))$, $q=\frac{p(n-1)}{n-p}\in [n-1,n+1]$, and 
\begin{equation}
\label{toprove}
\|W_1(\Delta_g-z)^{-1} W_2\|_{\mathcal{C}_{q}(L^2(M))}\le C |z|^{-1+\frac{n}{2p}} \|W_1\|_{L^{2p}(M)} \|W_2\|_{L^{2p}(M)}.
\end{equation}
\end{thm}

\begin{rem} When $z\in (0,+\infty)$, the operator in \eqref{toprove} may be taken to be either the outgoing or incoming resolvent $(\Delta_g-(z\pm i0))^{-1}$.
\end{rem}

In what follows we shall write $ E_{\sqrt{\Delta_g}}(\lambda)=1_{(-\infty,\lambda)}(\sqrt{\Delta_g})$, $\lambda>0$,
for  the spectral projection of $\sqrt{\Delta_g}$, and remark that the spectral measures $d(E_{\sqrt{\Delta_g}}(\lambda) u, u)_{L^2(M)}$ are absolutely continuous with respect to the Lebesgue measure, for any $u\in L^2(M)$.  Let us write 
\[
dE_{\sqrt{\Delta_g}}(\lambda):=\frac{d}{d\lambda} E_{\sqrt{\Delta_g}}(\lambda).
\]
The proof of Theorem \ref{thm_resolvent_Schatten_laplacian} is based on the following weighted Schatten norm estimates on the spectral measure $dE_{\sqrt{\Delta_g}}(\lambda)$ of $\sqrt{\Delta_g}$,  which extend the corresponding estimates of   Frank and Sabin \cite[Theorem 2]{Frank_Sabin}, obtained in the Euclidean setting. We believe that these estimates may be of some independent interest. 

\begin{thm}
\label{prop:spectral measure Schatten estimate}
Let $(M,g)$ be an asymptotically conic non-trapping manifold of dimension $n\ge 3$. Let $p\in [1, \frac{n+1}{2}]$. 
Then there exists $C>0$  such that  for all $\lambda>0$ and all $W_1,W_2 \in L^{2p}(M)$, we have $W_1 dE_{\sqrt{\Delta_g}}(\lambda) W_2\in \mathcal{C}_{q}(L^2(M))$, $q=\frac{p(n-1)}{n-p}\in [1,n+1]$, and 
\begin{equation}
\| W_1 dE_{\sqrt{\Delta_g}}(\lambda) W_2 \|_{\mathcal{C}_q (L^2(M))} \leq C \lambda^{-1 + \frac{n}{p}} \| W_1 \|_{L^{2p}(M)} \| W_2 \|_{L^{2p}(M)}. 
\label{spec-meas-bound}\end{equation}
\end{thm}

\begin{rem}  If the non-trapping assumption is dropped, the estimates in Theorem~\ref{prop:spectral measure Schatten estimate}, and therefore also Theorem~\ref{thm_resolvent_Schatten_laplacian}, may fail. Instead, the estimates will hold for all $\lambda \leq \lambda_0$ for a constant $C$ which depends on $\lambda_0$. A ``metric bottle'' example illustrating this, for which the best $C(\lambda_0)$ grows exponentially in $\lambda_0$, is given in \cite[Remark 8.8]{GHS}. 
\end{rem}

Let us now consider the Schr\"odinger operator $\Delta_g +V$ with a complex valued potential $V\in L^{p}(M)$, $\frac{n}{2}\le p<\infty$. As explained in Section \ref{sec_4_individual}, this operator has a natural $m$-sectorial realization on $L^2(M)$, and the spectrum of  $\Delta_g +V$ in $\C\setminus [0,\infty)$ consists of isolated eigenvalues of finite algebraic multiplicity.  

As an application of Theorem \ref{thm_GuHa}, we have the following generalization of the results of Frank \cite{Frank_2011}, \cite{Frank_2015}, and Frank and Simon \cite{Frank_Simon} concerning Keller type bounds on the individual eigenvalues of non-self-adjoint Schr\"odinger operators in the Euclidean setting to that of an asymptotically conic non-trapping manifold, see also \cite{FaKreVe}. 
\begin{thm}
\label{thm_indiviual_asymptotically conic}
Let $(M,g)$ be an asymptotically conic non-trapping manifold of dimension $n\ge 3$.  
\begin{itemize}
\item[(i)] Let $V\in L^{\gamma+\frac{n}{2}}(M)$ for some $0< \gamma\le \frac{1}{2}$. Then any eigenvalue $\lambda\in \C$ of the operator $\Delta_g+V$ satisfies 
\begin{equation}
\label{eq_8_0}
|\lambda|^\gamma\le C_{\gamma, n} \|V\|_{L^{\gamma+\frac{n}{2}}(M)}^{\gamma+\frac{n}{2}},
\end{equation}
where the constant $C_{\gamma, n} >0$ depends on $\gamma$ and $n$ only. 
\item[(ii)] If $V\in L^{\frac{n}{2}}(M)$ is such that  $\|V\|_{L^{\frac{n}{2}}(M)}$ is sufficiently small, then the operator $\Delta_g+V$ has no eigenvalues. 

\item[(iii)] Let $V\in L^{\gamma+\frac{n}{2}}(M)$ for some $\gamma> \frac{1}{2}$. Then any eigenvalue $\lambda\in \C$ of the operator $\Delta_g+V$ satisfies 
\begin{equation}
\label{eq_8_0_iii}
d(\lambda)^{\gamma-\frac{1}{2}}|\lambda|^{\frac{1}{2}}\le C_{\gamma, n}\|V\|^{\gamma+\frac{n}{2}}_{L^{\gamma+\frac{n}{2}}(M)},
\end{equation}
where $d(\lambda)$ is given by \eqref{eq_intro_2_new_d} and the constant $C_{\gamma, n} >0$ depends on $\gamma$ and $n$ only. 
\end{itemize}
\end{thm}

\begin{rem} Parts (i) and (ii) of Theorem \ref{thm_indiviual_asymptotically conic}
 have been established   in \cite[Proposition 7.2 ]{GuHa},   without specifying the radius of the disk, containing the eigenvalues of $\Delta_g+V$, in part (i).   
 \end{rem}

As a consequence of Theorem \ref{thm_resolvent_Schatten_laplacian}, we obtain the following analog of a result of Frank and Sabin \cite[Theorem 16]{Frank_Sabin}, concerning Lieb-Thirring type inequalities for the sums of eigenvalues of $\Delta_g+V$ in the case of a short range potential  $V\in L^p(M)$,  $p=\frac{n}{2}+\gamma$, where  $0\le \gamma\le \frac{1}{2}$. 
\begin{thm}
\label{thm_main_sums_asymp} 
Let $(M,g)$ be an asymptotically conic non-trapping manifold of dimension $n\ge 3$, and let $V\in L^p(M)$ with $p$ such that 
\[
\frac{n}{2}\le p\le \frac{n+1}{2}. 
\]
Let us denote by $\lambda_j$  the eigenvalues of $\Delta_g+V$ in $ \C\setminus [0,\infty)$, repeated according to their algebraic multiplicities. 
The following estimates then hold:  
 \begin{itemize}
\item[(i)] If $p=\frac{n}{2}$, we have  
\begin{equation}
\label{eq_11_0_asym}
\sum_{j}  \frac{\emph{\Im}\sqrt{\lambda_j}}{1+|\lambda_j|}< \infty,
\end{equation}
where the branch of the square root is chosen to have positive imaginary part. 

\item[(ii)] If $\frac{n}{2}<p \le \frac{n+1}{2}$, then 
\begin{equation}
\label{eq_11_0_1_asym}
\sum_{j}  \frac{d(\lambda_j)}{|\lambda_j|^{(1-\varepsilon)/2}}\le C_{\varepsilon,p,n} \|V\|_{L^p(M)}^{\frac{(1+\varepsilon) p}{2p-n}},
\end{equation}
for all $\varepsilon$ satisfying 
\[
\begin{cases} 
\varepsilon\ge 0, & \frac{n}{2}<p< \frac{n^2}{2n-1},\\
\varepsilon>\frac{p(2n-1)-n^2}{n-p}\ge  0, & \frac{n^2}{2n-1} \le p\le \frac{n+1}{2}.
\end{cases}
\]
\end{itemize}

\end{thm}

\begin{rem} If $\frac{n}{2}<p\le \frac{n+1}{2}$, then by Theorem \ref{thm_indiviual_asymptotically conic}  we know that the eigenvalues of 
$\Delta_g+V$ are confined to an open disk centered at the origin. Furthermore, it follows from \eqref{eq_11_0_1_asym}  that if a sequence of eigenvalues  $\C\setminus [0,\infty)\ni \lambda_{j_k}\to E>0$ then $\Im \lambda_{j_k}\in \ell^1$. In the case $p=\frac{n}{2}$ the bound \eqref{eq_11_0_asym}  controls a possible accumulation rate of eigenvalues in $\C\setminus [0,\infty)$ at infinity, and it implies in particular with the help of 
\[
\Im (\sqrt{\lambda})=\frac{|\Im \lambda|}{\sqrt{2(|\lambda|+\Re \lambda)}}
\] 
that if a sequence of eigenvalues  $\C\setminus [0,\infty)\ni \lambda_{j_k}\to E>0$ then $\Im \lambda_{j_k}\in \ell^1$. 
\end{rem}

As another application of the Schatten class estimates for the resolvent of $\Delta_g$ given in Theorem \ref{thm_resolvent_Schatten_laplacian}, we get the following generalization of a result of Frank \cite[Theorem 1.2]{Frank_2015}, concerning Lieb-Thirring type inequalities for the sums of eigenvalues $\Delta_g+V$ in the case of a long range potential  $V\in L^p(M)$,    $p=\gamma+\frac{n}{2}$,   $\gamma> \frac{1}{2}$. 

\begin{thm}
\label{thm_sums_long_range_asym}
Let $(M,g)$ be an asymptotically conic non-trapping manifold  of dimension $n\ge 3$, and let $V\in L^{p}(M)$ with $p=\gamma+\frac{n}{2}$,   $\gamma> \frac{1}{2}$. Then the eigenvalues $\lambda_j\in \C\setminus[0,\infty)$ of $\Delta_g+V$, repeated according to their algebraic multiplicities, satisfy the following bounds, for any $\varepsilon>0$, 
\[
\bigg(  \sum_{|\lambda_j|^\gamma \le C_{\gamma, n}\int_{M}|V|^{\gamma+\frac{n}{2}}dx} d(\lambda_j)^{2\gamma+\varepsilon}  \bigg)^{\frac{\gamma}{2\gamma+\varepsilon}}\le L_{\varepsilon, \gamma, n} \int_{M}|V|^{\gamma+\frac{n}{2}}dx,
\]
and for any $\varepsilon>0$, $0<\varepsilon'<\frac{\gamma}{\gamma+\frac{n}{2}},$ and $\mu\ge 1$, 
\begin{align*}
\bigg( \sum_{|\lambda_j|^\gamma \ge \mu C_{\gamma,n} \int_{M}|V|^{\gamma+\frac{n}{2}}dx}\frac{d(\lambda_j)^{2\gamma +\varepsilon}}{|\lambda_j|^{2\gamma-\frac{\gamma}{\gamma+\frac{n}{2}}+\varepsilon+\varepsilon'}}& \bigg)^{\frac{\gamma(\gamma+\frac{n}{2})}{\gamma-\varepsilon'(\gamma+\frac{n}{2})}} 
\le L_{\varepsilon,\varepsilon',\gamma,n} \mu^{-\frac{\varepsilon'(\gamma+\frac{n}{2})}{\gamma-\varepsilon'(\gamma+\frac{n}{2})}}\int_{M}|V|^{\gamma+\frac{n}{2}}dx.
\end{align*}
\end{thm}

\begin{rem} As observed in \cite{Frank_2015}, Theorem \ref{thm_sums_long_range_asym} has the following consequence: let $\gamma>1/2$ and $V\in L^{\gamma+n/2}(M)$. If $(\lambda_j)_{j=1}^\infty$
is a sequence of eigenvalues of $\Delta_g+V$ with $\lambda_j\to \lambda_0\in [0,\infty)$ then $\Im \lambda_j\in l^p$ for any $p>2\gamma$. 
\end{rem}

\begin{rem}  Let us emphasize once more that all our results, Theorems \ref{thm_resolvent_Schatten_laplacian},
\ref{prop:spectral measure Schatten estimate}, \ref{thm_indiviual_asymptotically conic}, 
\ref{thm_main_sums_asymp}, \ref{thm_sums_long_range_asym},  are valid for the metric Schr\"odinger operator in the Euclidean space $\R^n$, with a metric being a
non-trapping short range perturbation of the  Euclidean one, in the sense of 
\eqref{eq_short_range}. In particular, the results hold true for  the metric Schr\"odinger operator in the Euclidean space $\R^n$, with a metric being sufficiently small compactly supported perturbation of the  Euclidean one.
\end{rem}
 
\subsection{Outline of the paper}
The plan of the paper is as follows.  In Section \ref{sec_strategy_main_0}  we present our strategy for proving Theorem \ref{thm_resolvent_Schatten_laplacian}, which is the main result of the paper. 
Section \ref{sec:spectral measure estimates} is devoted to the proof of Theorem  \ref{prop:spectral measure Schatten estimate}, giving Schatten norm estimates on the spectral measure. In Section \ref{sec_consequences} we derive some Schatten norm estimates on the resolvent of the Laplacian, as a direct consequence of the Schatten norm estimates on the spectral measure and give their analogues at the endpoint case $p=\frac{n}{2}$, needed in the proof of Theorem  \ref{thm_resolvent_Schatten_laplacian}. 
The principal step in the proof of Theorem  \ref{thm_resolvent_Schatten_laplacian}, corresponding to the estimates on the spectrum, is carried out in Section \ref{sec_on_spectrum}.  Section \ref{sec_4_individual} contains the proof of Theorem \ref{thm_indiviual_asymptotically conic}, which follows the arguments of \cite{Frank_2015} and \cite{Frank_Simon} closely, relying on Theorem \ref{thm_GuHa}, with some small adjustments due to the fact that we are no longer in the Euclidean setting.  Finally, we observe in Section \ref{sec_bounds_sum_higher}  that Theorem \ref{thm_main_sums_asymp} and Theorem \ref{thm_sums_long_range_asym}  are direct consequences of Theorem \ref{thm_resolvent_Schatten_laplacian} combined with the arguments of  \cite[Theorem 16]{Frank_Sabin} and \cite[Theorem 1.2]{Frank_2015}. Appendix \ref{app} contains the proof of Lemma \ref{Guillarmou_Hassell_remark}, needed in the main text. Appendix \ref{app_2} is concerned with the analysis of the microlocal structure of the spectrally localized outgoing and incoming resolvent, used in the proof of Theorem  \ref{thm_resolvent_Schatten_laplacian}.


\section{Strategy of the proof of Theorem~\ref{thm_resolvent_Schatten_laplacian} }

\label{sec_strategy_main_0}

\subsection{Schatten norm estimates}\label{schattennormest}
We first recall the definition of the Schatten spaces of operators on $L^2(M)$, see \cite{Simon_trace}. Let $A$ be a compact operator on $L^2(M)$, and let  $\mu_j(A)$ be the singular values of $A$, given by  $\mu_j(A)=\lambda_j((A^*A)^{1/2})$. Here $\lambda_j(B)$ denotes the eigenvalues of a positive  self-adjoint compact operator $B$, arranged in decreasing order.  The  Schatten norm of $A$ of order $1\le q<\infty$  is   defined as follows,
\[
\|A\|^q_{\mc{C}_{q}(L^2(M))}=\sum_{j=1}^\infty \mu_j(A)^q=\text{tr} ((A^*A)^{q/2}). 
\]

The basic mechanism for proving Schatten norm estimates of Theorem \ref{thm_resolvent_Schatten_laplacian} and Theorem \ref{prop:spectral measure Schatten estimate} comes from the fact that the Schatten spaces are complex interpolation spaces, see \cite[Theorem 2.9]{Simon_trace},   \cite[p. 154]{Simon_operator_theory}, and  from the following result of Frank and Sabin \cite[Proposition 1]{Frank_Sabin}. 
\begin{prop}
\label{prop:Frank-Sabin}
Let $T_s$ be an analytic family of operators, defined on the strip $\{ s \in \CC \mid -\lambda_0 \leq \emph{\Re} s \leq 0 \}$ for some $\lambda_0 > 1$, acting on functions on $M$. Assume that we have operator norm bounds
\[
\| T_{ir} \|_{L^2(M) \to L^2(M)} \leq M_0 e^{a|r|}, \quad \| T_{-\lambda_0 + ir} \|_{L^1(M) \to L^\infty(M)} \leq M_1 e^{a|r|} \quad \forall r \in \RR
\]
for some $a\ge 0$ and  $M_0, M_1 > 0$. Then for any $W_1, W_2 \in L^{2\lambda_0}(M)$, the operator $W_1 T_{-1} W_2$ belongs to the Schatten class $\mathcal{C}_{2\lambda_0}(L^2(M))$ and we have the estimate
\[
\| W_1 T_{-1} W_2 \|_{\mathcal{C}_{2\lambda_0}} \leq M_0^{1 - \frac{1}{\lambda_0}} M_1^{\frac{1}{\lambda_0}} \| W_1\|_{L^{2\lambda_0}(M)}  \| W_2\|_{L^{2\lambda_0}(M)}.
\]
\end{prop}

Let us recall briefly the proof of Proposition \ref{prop:Frank-Sabin}.  The result is established by considering the analytic family of operators $S_s = |W_1|^{-1-s} W_1 T_s W_2 |W_2|^{-1-s}$. 
This family has the property that $S_{-1} = W_1 T_{-1} W_2$ and it satisfies the following estimates on the boundary of the strip. For $s = ir$, $r$ real, we have 
$$
\| S_{ir} \|_{L^2(M) \to L^2(M)} \le  \| T_{ir} \|_{L^2(M) \to L^2(M)} \leq M_0 e^{a|r|},
$$
and for $s = -\lambda_0 + ir$, we note that $T_s$ has its Schwartz kernel bounded pointwise by $M_1 e^{a|r|}$ (due to the $L^1\to L^\infty$ bound on $T_s$) and $|W_1|^{-s}$, $|W_2|^{-s}$ are $L^2$ functions, hence $S_s$ is a Hilbert-Schmidt operator  with the Hilbert-Schmidt norm  bounded by $M_1 e^{a|r|} \| W_1 \|_{L^{2\lambda_0}(M)}^{\lambda_0} \| W_2 \|_{L^{2\lambda_0}(M)}^{\lambda_0}$. Interpolating between the operator norm and the Hilbert-Schmidt norm gives us a bound on the Schatten norms, in particular at $s = -1$, where we obtain the Schatten norm at exponent $2\lambda_0$.

\subsection{Strategy}
\label{sec_strategy_main}

The principal idea of the proof of the Euclidean analog of  Theorem~\ref{thm_resolvent_Schatten_laplacian}, which is due to Frank and Sabin \cite[Theorem 12]{Frank_Sabin}, is to establish the following pointwise bound for the Schwartz kernel of the powers of the resolvent $(\Delta-z)^{-\alpha}$, \begin{equation}
\label{eq_int_pointwise_res}
|(\Delta-z)^{-\alpha}(x,y)|  \leq Ce^{C(\text{Im}(\alpha))^2} |z|^{\frac{n-1}{4} - \frac{\text{Re} (\alpha)}{2}}|x-y|^{{\rm Re}(\alpha)-\frac{n+1}{2}}, \quad x,y\in \R^n.
\end{equation}
Here $z\in \C\setminus [0,\infty)$,  $\alpha\in \C$, $\text{Re}(\alpha)\in [\frac{n-1}{2},\frac{n+1}{2}]$. The desired Schatten bound \eqref{toprove} in the Euclidean case is therefore a consequence of \eqref{eq_int_pointwise_res} combined with the H\"older and Hardy--Littelewood--Sobolev inequalities as well as an interpolation argument. 

Unfortunately, the natural analog of the pointwise bound \eqref{eq_int_pointwise_res} does not hold in general, for $z$ close to the spectrum of $\Delta_g$, for asymptotically conic manifolds, essentially because there can be conjugate points for the geodesic flow, and to prove the bound  \eqref{toprove} we have to proceed differently. 

Our strategy of the proof of Theorem ~\ref{thm_resolvent_Schatten_laplacian} is to establish the Schatten norm estimate \eqref{toprove}  for $W_1 (\Delta_g - z)^{-1} W_2$ for $z$ on the negative real axis, and for $z$ just above and below the spectrum, that is, for $W_1 (\Delta_g - (z \pm i0))^{-1} W_2$, for $z > 0$. We then use the Phragm\'en-Lindel\"of theorem to obtain the result on the whole of the complex plane, excluding the origin.  

Let us give the proof of Theorem~\ref{thm_resolvent_Schatten_laplacian}, assuming that it has been established for $z< 0$ and for $z \pm i0$, $z > 0$. Let $W_1, W_2\in L^{2p}(M)$ with $p\in [\frac{n}{2},\frac{n+1}{2}]$,  and let us consider the following bilinear form for $z\in \C\setminus [0,\infty)$, 
\begin{equation}
\label{eq_Bz}
B_z(W_1,W_2):=W_1(\Delta_g-z)^{-1}W_2.
\end{equation}
When $z\in (0,\infty)$, we extend the definition of $B_z$ by taking the outgoing resolvent $(\Delta_g-(z+i0))^{-1}$ in \eqref{eq_Bz}. Thus, we know that  for $z \in \RR \setminus \{ 0 \}$, $B_z$ is a bounded bilinear form
\[
B_z: L^{2p}(M)\x L^{2p}(M)\to \mc{C}_{q}(L^2(M)), \quad p\in \bigg[\ndemi,\frac{n+1}{2}\bigg], \quad q=\frac{p(n-1)}{n-p},
\]
such that 
\begin{equation}
\label{eq_Bz_1}
\| B_z(W_1,W_2)\|_{\mc{C}_{q}}\le C|z|^{-1+\frac{n}{2p}} \|W_1\|_{L^{2p}(M)}\|W_2\|_{L^{2p}(M)}.
\end{equation}
We now complete the proof of Theorem \ref{thm_resolvent_Schatten_laplacian} by a Phragm\'en-Lindel\"of argument. In doing so,  let $W_1,W_2\in C_0^\infty(M)$.  We claim that the function $H(z):= B_z(W_1,W_2)$
is holomorphic in  $\Im z>0$ with values in $\mc{C}_q(L^2(M))$  such that 
\[
\| H(z)\|_{\mc{C}_q}\le C(|z|^{-1/2}+|z|^{1/2}).
\]
Indeed, for $\Im  z >0$, the operator $W_1(\Delta_g-z)^{-1}W_2: L^2(M)\to H^2(M)\cap \mathcal{E}'(K)$ is bounded  where $K$ is a compact set containing the support of  $W_1$. Furthermore, it depends holomorphically on $z$ with $\Im z>0$, and satisfies the bound 
\[
\|W_1(\Delta_g-z)^{-1}W_2 \|_{\mathcal{L}(L^2(M), H^2(M))}\le C( |z|^{-1/2} + |z|^{1/2}), \quad \Im z\ge 0, \quad z\ne 0,  
\]
see \cite{Melrose} for intermediate values of $z$, \cite{Vasy_Zworski_2000} for $|z| \to \infty$ and \cite[Prop. 1.26]{RoTa} for $|z| \to 0$. Now the embedding $H^2(M)\cap \mathcal{E}'(K)\to L^2(M)$ is an operator in $\mc{C}_{n/2+\eps}$ for all $\eps>0$ in view of the Weyl law for the Laplacian on a compact manifold. Since $q>n/2$, we deduce the claim. 

The function $H(z)$ is continuous for $\Im z\ge 0$, $z\ne 0$, with values in $\mc{C}_q(L^2(M))$ and to avoid the problem at $z=0$, we consider the map
\[
F(z):=\cjg H(e^{z}),T\cjd e^{(1-\frac{n}{2p})z}
\]
for a fixed $T\in \mc{C}_{q'}(L^2(M))$ with norm $\|T\|_{\mc{C}_{q'}}=1$.  Here $\frac{1}{q'}+\frac{1}{q}=1$ and  the product is the duality pairing between the Banach space $\mc{C}_{q}$ and its dual $\mc{C}_{q'}$.
 Then $F(z)$ is holomorphic in $\Im z\in (0,\pi)$, continuous on the closure, and enjoys the bounds 
\[
\begin{gathered} 
|F(z)|\leq C e^{C|z|} \textrm{ for } 0\leq \Im z\leq \pi, \\ 
 |F(z)|\leq C  \|W_1\|_{L^{2p}(M)}\|W_2\|_{L^{2p}(M)}  \textrm{ for } \Im z\in \{0,\pi\}
\end{gathered}
\] 
in view of \eqref{eq_Bz_1}.  Applying the Phragm\'en-Lindel\"of principle,  we deduce 
that $|F(z)|\leq C\|W_1\|_{L^{2p}(M)}\|W_2\|_{L^{2p}(M)}$ for all $z\in \C$ such that $0\le \Im z\le \pi$,  and therefore,  
\[
\|H(z)\|_{\mc{C}_q}\leq  C|z|^{-1+\frac{n}{2p}}\|W_1\|_{L^{2p}(M)}\|W_2\|_{L^{2p}(M)}, \quad \Im z\ge 0, \quad z\ne 0. 
\]
By a density argument, we obtain the bound \eqref{toprove} for $\Im z\ge 0$, $z\ne 0$. By considering the adjoint of the operator $B_z$, we complete the proof of Theorem \ref{thm_resolvent_Schatten_laplacian}.

This argument reduces the problem to proving estimate \eqref{toprove} for $z \in \RR \setminus \{ 0 \}$. 
We find it convenient to first prove the corresponding estimate for the spectral measure given in Theorem \ref{prop:spectral measure Schatten estimate}.  The proof of Theorem \ref{prop:spectral measure Schatten estimate} relies crucially on the $TT^*$ structure of the spectral measure. 

When $z\in (-\infty,0)$ and $p\in (\frac{n}{2},\frac{n+1}{2}]$, the Schatten norm estimate \eqref{toprove} is a direct consequence of Theorem~\ref{prop:spectral measure Schatten estimate}, and at the endpoint case $p=\frac{n}{2}$,   the Schatten norm estimate \eqref{toprove} follows from the heat kernel estimates due to Grigor'yan \cite{Gr} and Varopoulos \cite{Va}.

Establishing the Schatten norm estimate \eqref{toprove}  for $W_1 (\Delta_g - (z \pm i0))^{-1} W_2$ with $z>0$ represents the main difficulty in the proof of Theorem  \ref{thm_resolvent_Schatten_laplacian}. When doing so, 
following  \cite{GuHa},  \cite{GHS}, and \cite{HaZh}, we use a microlocal partition of the identity $\sum_{i=1}^N Q_i(\eta)= \Id$, where $Q_i(\eta)$  are pseudodifferential operators depending on the energy parameter $0<\eta\sim |z|^{1/2}$, constructed  in \cite{GHS}.   Splitting up the operator $W_1 (\Delta_g - (z \pm i0))^{-1} W_2$ by means of the partition of the identity, we are led to estimate the individual terms $W_1Q_i(\eta)^*(\Delta_g - (z \pm i0))^{-1}Q_j(\eta)W_2$, and here the most interesting contributions arise when $i=j$.  When handling those, we proceed by establishing pointwise bounds for the Schwartz kernel of the operator 
\[
Q_i(\eta)^*\phi\bigg(\frac{\Delta_g}{z}\bigg)(\Delta_g - (z \pm i0))^{-s}Q_j(\eta),\quad \text{Re}\,s\in \bigg[\frac{n-1}{2},\frac{n+1}{2}\bigg],
\]
analogous to the Euclidean estimates \eqref{eq_int_pointwise_res}. Here $\phi$ is a cut-off near $1$.


\section{Schatten norm estimates on the spectral measure. Proof of Theorem \ref{prop:spectral measure Schatten estimate}}

\label{sec:spectral measure estimates}

Our starting point for the proof is the operator partition of unity, $\Id = \sum_{i=1}^N Q_i(\eta)$,  depending on $\eta > 0$, constructed  in \cite{GHS}. This partition of unity enjoys the following estimates, in particular: there exists $\delta>0$  sufficiently small but fixed such that for all $k=0,1,2,\dots,$ there is $C_k>0$ such that for all $m,m'\in M$, we have 
\begin{equation}
\label{eq_sec_measure_1} 
\begin{aligned}
\Big|\pl_{\la}^k\big(Q_i(\eta)^*dE_{\sqrt{\Delta_g}}(\la)Q_i(\eta)\big)(m,m')\Big|\leq C_k\la^{n-1-k}(1+\la d(m,m'))^{-\frac{(n-1)}{2}+k}, \\ \lambda\in [(1-\delta)\eta,(1+\delta)\eta],
\end{aligned}
\end{equation}
with $d(\cdot,\cdot)$ being the Riemannian distance on $M$.  We say more about this partition of the identity in Section \ref{sec:partition of unity} below;  here, we can use results of  \cite{Chen} and \cite{GHS} as a `black box'.   Then for all  $\lambda\in [(1-\delta/2)\eta,(1+\delta/2)\eta]$,    we use the partition of unity to decompose the spectral measure sandwiched between two $L^{2p}$ functions: 
\begin{equation}
W_1 dE_{\sqrt{\Delta_g}}(\lambda)W_2 = \sum_{i, j = 1}^N W_1 Q_i(\eta)^* dE_{\sqrt{\Delta_g}}(\lambda) Q_j(\eta) W_2.
\label{Q-decomp}\end{equation}

Let $p \in [1, \frac{n+1}{2}]$ and $q = \frac{p(n-1)}{n-p}\in [1,n+1]$. 
In the first step, we shall prove microlocalized estimates of the form 
\begin{equation}
\label{spec-meas-bound-QiQi}
\| W_1 Q_i(\eta)^* dE_{\sqrt{\Delta_g}}(\lambda) Q_i(\eta) W_2 \|_{\mathcal{C}_q} \leq C \lambda^{-1 + \frac{n}{p}} \| W_1 \|_{L^{2p}(M)} \| W_2 \|_{L^{2p}(M)},
\end{equation}
for the diagonal ($i=j$) terms of the decomposition \eqref{Q-decomp}. In doing so, we shall follow  \cite[Proof of Theorem 2]{Frank_Sabin} and start by showing \eqref{spec-meas-bound-QiQi} at the endpoints $p=\frac{n+1}{2}$ and $p=1$, i.e. 
\begin{equation}
\label{spec-meas-bound-QiQi_end_1}
\| W_1 Q_i(\eta)^* dE_{\sqrt{\Delta_g}}(\lambda) Q_i(\eta) W_2 \|_{\mathcal{C}_{n+1}} \leq C \lambda^{\frac{n-1}{n+1}} \| W_1 \|_{L^{n+1}(M)} \| W_2 \|_{L^{n+1}(M)},
\end{equation}
and 
\begin{equation}
\label{spec-meas-bound-QiQi_end_2}
\| W_1 Q_i(\eta)^* dE_{\sqrt{\Delta_g}}(\lambda) Q_i(\eta) W_2 \|_{\mathcal{C}_{1}} \leq C \lambda^{n-1} \| W_1 \|_{L^{2}(M)} \| W_2 \|_{L^{2}(M)},
\end{equation}
respectively. Once the estimates \eqref{spec-meas-bound-QiQi_end_1} and \eqref{spec-meas-bound-QiQi_end_2} have been established,  the bound \eqref{spec-meas-bound-QiQi} follows by a complex interpolation argument  applied to the analytic family of operators $\zeta\mapsto W_1^{\frac{2}{n+1}+\zeta\frac{n-1}{n+1}} Q_i(\eta)^* dE_{\sqrt{\Delta_g}} (\lambda)Q_i(\eta)W_2^{\frac{2}{n+1}+\zeta\frac{n-1}{n+1}}$ in the strip $0\le \Re \zeta\le 1$,  with $W_j\ge 0$ being simple functions such that $\|W_j\|_{L^2(M)}=1$, $j=1,2$, see \cite[Theorem 2.9]{Simon_trace}.

Now to prove the estimate \eqref{spec-meas-bound-QiQi_end_1}, we shall consider the following family of operators,  
\[
T_s := Q_i(\eta)^* \phi\bigg(\frac{\sqrt{\Delta_g}}{\lambda}\bigg) \chi_+^s(\lambda - \sqrt{\Delta_g}) Q_i(\eta), \quad -\frac{(n+1)}{2} \leq \Re s \leq 0, 
\]
introduced in \cite{Chen} and \cite[Definition 3.2]{GHS}.  Here  $\phi\in C^\infty_0((1-\delta/4,1+\delta/4))$ is such that $\phi(t)=1$ in a neighborhood of $t=1$,  and  $\chi_+^s$ is the family of distributions on $\R$, entire analytic  in $s\in \C$ and such that 
\[
\chi_+^s(\lambda)=\frac{\lambda_+^s}{\Gamma(s+1)}, \quad \Re s>-1,
\]
where $\lambda_+=\max(\lambda,0)$, see \cite[Section 3.2]{Hormander_books_1}. Note that at least formally, we have
\[
\chi^0_+(\lambda-\sqrt{\Delta_g})=E_{\sqrt{\Delta_g}}(\lambda), \quad \chi_+^{-k}(\lambda-\sqrt{\Delta_g})=\bigg(\frac{d}{d\lambda}\bigg)^{k-1}dE_{\sqrt{\Delta_g}}(\lambda),\quad k=1,2,\dots.
\]
Recall from \cite[Definition 3.2]{GHS} that $T_s$ is the operator whose Schwartz kernel is given by 
\begin{equation}
\label{eq_def_T_s}
\begin{aligned}
\big(Q_i(\eta)^* \phi\bigg(\frac{\sqrt{\Delta_g}}{\lambda}\bigg) &\chi_+^s(\lambda - \sqrt{\Delta_g}) Q_i(\eta)\big)(m,m')\\
&=
\int\chi_+^{k+s}(\lambda-\mu)\p_\mu^k\bigg (Q_i(\eta)^*\phi \bigg(\frac{\mu}{\lambda}\bigg) dE_{\sqrt{\Delta_g}}(\mu) Q_i(\eta)\bigg)(m,m')d\mu,
\end{aligned}
\end{equation}
where $k\in \N$ is such that $\Re s+k>-1$. As $\mu\in [\eta(1-\delta), \eta(1+\delta)]$ for $\lambda\in [(1-\delta/2)\eta,(1+\delta/2)\eta]$ and $\mu/\lambda\in \supp(\phi)$, thanks to the estimates \eqref{eq_sec_measure_1} the integral in \eqref{eq_def_T_s} is well defined.

As explained in \cite{GHS}, the family of operators $T_s $ is analytic in the sense of Stein in the  strip $-\frac{(n+1)}{2} \leq \Re s \leq 0$. When $\Re s = 0$,  we have
\[
\| T_s\|_{L^2(M)\to L^2(M)}\le C e^{\frac{\pi |s|}{2}},
\]
and relying on the estimates \eqref{eq_sec_measure_1} it was shown in  \cite{Chen} and  \cite{GHS}  that when $\Re s = -\frac{(n+1)}{2}$, we have
\[
\| T_s \|_{L^1(M) \to L^\infty(M)} \leq C (1 + |r|) e^{\frac{\pi |r|}{2}} \lambda^{\frac{n-1}{2}}, \quad s = -\frac{(n+1)}{2} + ir, \ r \in \RR.
\]
Applying Proposition~\ref{prop:Frank-Sabin}, we get, for  any two complex valued functions $W_1, W_2\in L^{n+1}(M)$,
\begin{align*}
W_1 T_{-1} W_2 &= W_1 Q_i(\eta)^* \phi\bigg(\frac{\sqrt{\Delta_g}}{\lambda}\bigg) \chi_+^{-1}(\lambda - \sqrt{\Delta_g}) Q_i(\eta) W_2\\
& = W_1 Q_i(\eta)^* dE_{\sqrt{\Delta_g}}(\lambda) Q_i(\eta) W_2
\end{align*}
is in the Schatten $\mathcal{C}_{n+1}$ class and \eqref{spec-meas-bound-QiQi_end_1} holds.

To show \eqref{spec-meas-bound-QiQi_end_2}, we recall from \cite{GHS} that we have a pointwise kernel bound on the (microlocalized) spectral measure, 
\begin{equation}
\| Q_i(\eta)^* dE_{\sqrt{\Delta_g}}(\lambda) Q_i(\eta) \|_{L^1(M) \to L^\infty(M)} \leq C \lambda^{n-1}.
\label{sm-Poisson}\end{equation}
Also, we have 
\begin{equation}
\label{eq_spectral_measure_dec}
dE_{\sqrt{\Delta_g}}(\lambda) = (2\pi)^{-1} P(\lambda) P^*(\lambda), 
\end{equation}
where $P(\lambda):L^2(\p M)\to L^{r}(M)$, $r\in [\frac{2(n+1)}{n-1},\infty]$,  is the Poisson operator, see \cite{GHS}.  
Using the $T^*T$ trick, it follows from \eqref{sm-Poisson} and \eqref{eq_spectral_measure_dec} that  
\[
\|Q_i(\eta)^* P(\lambda)\|_{L^2(\partial M) \to L^\infty(M)}\le C \lambda^{\frac{n-1}{2}}. 
\]
The Schwartz kernel $Q_i(\eta)^* P(\lambda)(m,m')$ of the operator $Q_i(\eta)^* P(\lambda)$ satisfies therefore,  
\[
\|Q_i(\eta)^* P(\lambda)(m,\cdot) \|_{L^2(\p M)}\le C \lambda^{\frac{n-1}{2}}
\]
for almost all $m\in M$.  Thus, for any $W_1 \in L^2(M)$, the operator $W_1 Q_i(\eta)^* P(\lambda):L^2(\p M)\to L^2(M)$ is Hilbert-Schmidt with the norm bounded by $ C \lambda^{\frac{n-1}{2}} \| W_1 \|_{L^2(M)}$. Taking adjoints, we find that $P(\lambda)^* Q_i(\eta) W_2$ is a Hilbert-Schmidt operator  with norm bounded by $C \lambda^{\frac{n-1}{2}} \| W_2 \|_{L^2(M)}$. Therefore, $(2\pi)^{-1}$ times the composition of these two operators, which is precisely $W_1 Q_i(\eta)^* dE_{\sqrt{\Delta}}(\lambda) Q_i(\eta) W_2$, is of trace class and \eqref{spec-meas-bound-QiQi_end_2} follows.

In the second step, we shall bound the Schatten norm of the off-diagonal ($i \neq j$) terms  in the decomposition \eqref{Q-decomp}, i.e. we shall prove the following estimate, 
\begin{equation}
\label{spec-meas-bound-QiQj}
\| W_1 Q_i(\eta)^* dE_{\sqrt{\Delta_g}}(\lambda) Q_j(\eta) W_2 \|_{\mathcal{C}_q} \leq C \lambda^{-1 + \frac{n}{p}} \| W_1 \|_{L^{2p}(M)} \| W_2 \|_{L^{2p}(M)}.
\end{equation}
 As above, we shall exploit the $ T^*T$ structure of the spectral measure. 

Let $T: L^2(M)\to L^2(\p M)$ be a compact operator and $q\ge 1$. Then $T^*T\in \mathcal{C}_q(L^2(M))$  if and only if $T\in \mathcal{C}_{2q}(L^2(M),L^2(\p M))$, and moreover, $\| T^* T\|_{ \mathcal{C}_q}=\|T\|_{ \mathcal{C}_{2q}}^2$. This is a consequence of the following equality for the singular values,
\begin{equation}
\label{eq_singular_values_103}
\mu_k(T^*T)=\mu_k(T)^2.
\end{equation}

Moreover, if $T_1, T_2$ are in $\mathcal{C}_{2q}(L^2(M),L^2(\p M))$, then $T_1^* T_2$ is in $\mathcal{C}_{q}(L^2(M))$, and 
\begin{equation}
\label{T^*T-inequality}
\| T_1^* T_2 \|_{ \mathcal{C}_q}^q \leq	\| T_1^* T_1 \|_{ \mathcal{C}_q}^q + \| T_2^* T_2 \|_{ \mathcal{C}_q}^q,
\end{equation}
see for example \cite{McCarthy}. 
Using \eqref{eq_spectral_measure_dec}, we write
\begin{equation}
\label{eq_decom_with_poten}
W_1 Q_i(\eta)^* dE_{\sqrt{\Delta_g}}(\lambda) Q_j(\eta) W_2  =(2\pi)^{-1} T_1^* T_2, 
\end{equation}
where $T_1 = P(\lambda)^* Q_i(\eta) \overline{W}_1$, and  $ T_2 = P(\lambda)^* Q_j(\eta) W_2$.
Now it follows from \eqref{spec-meas-bound-QiQi} that $T_1^*T_1 \in \mathcal{C}_q(L^2(M))$, $T_2^*T_2 \in \mathcal{C}_q(L^2(M))$, and we have
\begin{align*}
\|T_1^*T_1\|_{\mathcal{C}_q}\le C\lambda^{-1+\frac{n}{p}}\|W_1\|_{L^{2p}(M)}^2,\quad 
\|T_2^*T_2\|_{\mathcal{C}_q}\le C\lambda^{-1+\frac{n}{p}}\|W_2\|_{L^{2p}(M)}^2.
\end{align*}
By the discussion above, this is equivalent to the fact that $T_1\in C_{2q}(L^2(M), L^2(\p M))$
and  $T_2\in C_{2q}(L^2(M), L^2(\p M))$. It follows from \eqref{eq_decom_with_poten} and discussion above that $W_1 Q_i(\eta)^* dE_{\sqrt{\Delta_g}}(\lambda) Q_j(\eta) W_2\in \mathcal{C}_{q}(L^2(M))$, and using \eqref{T^*T-inequality}, we get that 
\[
\|W_1Q_i(\eta)^* dE_{\sqrt{\Delta_g}}(\lambda) Q_j(\eta) W_2\|_{\mathcal{C}_{q}}\le C\lambda^{-1+\frac{n}{p}}(\|W_1\|^2_{L^{2p}(M)} +\|W_2\|^2_{L^{2p}(M)}).
\]
Thus, \eqref{spec-meas-bound-QiQj} follows by bilinearity in $W_1,W_2$. This completes the proof of Theorem \ref{prop:spectral measure Schatten estimate}.

\section{Consequences of the spectral measure estimates for $p\in (\frac{n}{2},\frac{n+1}{2}]$ and their analogues at the endpoint $p=\frac{n}{2}$}

\label{sec_consequences}

\subsection{Consequences of the spectral measure Schatten norm estimate}

\label{subsections_cons_2}

Using Theorem~\ref{prop:spectral measure Schatten estimate} and Minkowski's integral inequality, we can deduce some Schatten estimates on the resolvent. In this subsection, we only treat the case $p > \frac{n}{2}$. 

The first result applies for $z$ in any sector excluding the positive real axis. 

\begin{prop}
\label{prop:resolvent-away-from-spectrum1} 
Let $p \in (\frac{n}{2}, \frac{n+1}{2}]$, and suppose $W_1, W_2\in L^{2p}(M)$.  Let $\epsilon > 0$ be arbitrary.  Then for $z \in \CC$ such that $z \neq 0, \arg z \in [\epsilon, 2\pi - \epsilon]$, the sandwiched resolvent $W_1 (\Delta_g - z)^{-1}W_2$ is in the Schatten class $\mathcal{C}_q (L^2(M))$ with $q = \frac{p(n-1)}{n-p} \in (n-1, n+1]$, and we have 
\[
\|W_1 (\Delta_g - z)^{-1}W_2\|_{\mathcal{C}_q}\le C |z|^{-1+\frac{n}{2p}}\| W_1 \|_{L^{2p}(M)} \| W_2 \|_{L^{2p}(M)},
\]
where $C$ depends on $p$, $\epsilon$ and $(M,g)$, but not $z$. 
\end{prop}

\begin{proof}
We express the operator $W_1 (\Delta_g - z)^{-1}W_2$ as 
$$
W_1 (\Delta_g - z)^{-1}W_2 = \int_0^\infty (\lambda^2 - z)^{-1} W_1 dE_{\sqrt{\Delta_g}}(\lambda) W_2 d\lambda.
$$
The result follows by estimating the Schatten norm of $W_1 dE_{\sqrt{\Delta_g}}(\lambda) W_2 $ using Theorem~\ref{prop:spectral measure Schatten estimate} and noting that provided $p > \frac{n}{2}$, we have
$$
\int_0^\infty |\lambda^2 - z|^{-1} \lambda^{-1+\frac{n}{p}} \, d\lambda \leq C |z|^{-1+\frac{n}{2p}},
$$
where $C$ depends on $p$ and $\epsilon$ but does not depend on $z$ in the given sector. 
\end{proof}

In a similar manner we obtain `elliptic' estimates on the resolvent, where we remove the singularity in the spectral multiplier. In this way we can obtain estimates on the positive real axis. To state these, we fix a function $\phi : [0, \infty) \to [0, 1]$ such that $\phi(t) = 1$ for $t$ in a neighbourhood of $t=1$, and has support in a slightly bigger neighborhood of $t=1$. 

\begin{prop}
\label{prop:resolvent-near-spectrum-away-from-singularity1} Let $p \in (\frac{n}{2}, \frac{n+1}{2}]$, and suppose $W_1, W_2\in L^{2p}(M)$. Then for $z  \in \C\setminus\{0\}$,  the operator $W_1 \big(1 - \phi\big)\big(\frac{\Delta_g}{|z|}\big) (\Delta_g - z)^{-1}W_2$ is in the Schatten class $\mathcal{C}_q(L^2(M))$ with $q = \frac{p(n-1)}{n-p} \in (n-1, n+1]$, and we have
\[
\bigg\|W_1 \bigg(1 - \phi\bigg)\bigg(\frac{\Delta_g}{|z|}\bigg) (\Delta_g - z)^{-1}W_2 \bigg\|_{\mathcal{C}_q}\le C |z|^{-1+\frac{n}{2p}}\| W_1 \|_{L^{2p}(M)} \| W_2 \|_{L^{2p}(M)},
\]
where $C$ depends on $p$ and on $(M,g)$, but not $z$. 
\end{prop}

\begin{proof}
Again we express the operator using an integral over the spectral measure, and estimate the Schatten norm of the spectral measure using Proposition~\ref{prop:spectral measure Schatten estimate} and Minkowski's integral inequality. This time we obtain the integral 
$$
\int_0^\infty |\lambda^2 - z|^{-1} \bigg(1 - \phi\bigg)\bigg(\frac{\lambda^2}{|z|}\bigg) \lambda^{-1+\frac{n}{p}} \, d\lambda 
$$
and it is straightforward to check that this is bounded by $C |z|^{-1+\frac{n}{2p}}$ uniformly in $z$. 
\end{proof}

\subsection{Analogues at the endpoint  $p = \frac{n}{2}$} 

In the case $p=\frac{n}{2}$,  the arguments used in the proofs of Propositions~\ref{prop:resolvent-away-from-spectrum1} and \ref{prop:resolvent-near-spectrum-away-from-singularity1} are no longer valid and need to be replaced. In view of the Phragm\'en--Lindel\"of  argument, explained in Section \ref{sec_strategy_main},  we only need to do this for $z$ negative in the case of Proposition~\ref{prop:resolvent-away-from-spectrum1} and $z$ positive in the case of Proposition~\ref{prop:resolvent-near-spectrum-away-from-singularity1}. To this end we prove the following two results. 

\begin{prop}
\label{negativez}
Let $p=\frac{n}{2}$. There is $C>0$ such that  for all $z<0$ and for all $W_1,W_2\in L^{n}(M)$,  the operator $W_1(\Delta_g-z)^{-1}W_2\in \mathcal{C}_{n-1}(L^2(M))$ and we have
\begin{equation}
\label{eq_negativez}
\|W_1(\Delta_g-z)^{-1}W_2 \|_{\mathcal{C}_{n-1}}\le C\|W_1\|_{L^n(M)}\|W_2\|_{L^n(M)}.
\end{equation}
\end{prop}

\begin{proof} 
Here we use a slight variation of Proposition~\ref{prop:Frank-Sabin}. Let $W_1,W_2$ be non-negative simple functions and consider the analytic  family of operators
\[
S_s=W_1^{-s} (\Delta_g - z)^s W_2^{-s}, \quad -\frac{(n-1)}{2}\le \Re s\le 0. 
\]
Clearly, when  $\Re s = 0$, we have
\begin{equation}
\label{eq_negativez_1}
\|S_s\|_{L^2(M) \to L^2(M)}\le C.
\end{equation}
Next, we will show that, when $\Re s = -\frac{(n-1)}{2}$, then $S_s$ is Hilbert-Schmidt and we have 
\begin{equation}
\label{eq_negativez_2}
\| S_s\|_{\mathcal{C}_2}\le C e^{C|\mathrm{Im} \, s|} \| W_1\|^{\frac{n-1}{2}}_{L^n(M)} \| W_2\|^{\frac{n-1}{2}}_{L^n(M)}.
\end{equation}
 This allows us to run the interpolation argument in the proof of Proposition~\ref{prop:Frank-Sabin}.

To prove \eqref{eq_negativez_2},  on the line $\Re s = -\frac{(n-1)}{2}$,  we express $(\Delta_g - z)^{s}$ in terms of the heat kernel: 
\begin{equation} 
\label{heatkernel_k_1}
\Gamma(-s) (\Delta_g-z)^{s}(m,m')= \int_0^\infty t^{-s-1}e^{tz} e^{-t\Delta_g}(m,m')dt.
\end{equation}
We now use heat kernel estimates. 
Due to Varopoulos \cite{Va}, we have the estimate 
$||e^{-t\Delta_g}||_{L^1\to L^\infty}\leq Ct^{-\ndemi}$ and by a result of Grigor'yan \cite{Gr}, this implies a pointwise upper Gaussian estimate on the heat kernel 
\begin{equation} 
\label{heatkernel}
|e^{-t\Delta_g}(m,m')|\leq Ct^{-\ndemi} e^{-\frac{cd(m,m')^2}{t}}, \quad t>0,  
\end{equation}
for some $c>0$. The integral in \eqref{heatkernel_k_1} is convergent for all $m\ne m'$ due to \eqref{heatkernel}. 
We thus get  for all $m\not=m'$ and $z\in(-\infty,0)$, and uniformly for all $s$ such that $\Re s = -\frac{(n-1)}{2}$, 
\begin{equation} 
\label{heatkernel_k_2}
\begin{split}
|\Gamma(-s) (\Delta_g-z)^{s}(m,m')|& \leq C\int_0^\infty t^{-\frac{3}{2}}e^{-\frac{cd(m,m')^2}{t}+z t} \, dt \\
& \leq C d(m,m')^{-1}\int_0^\infty t^{-\frac{3}{2}}e^{-\frac{c}{t}+zd(m,m')^2t}dt  \\
& \leq C d(m,m')^{-1}.
\end{split}
\end{equation}
Using H\"older's  inequality,  the generalized Hardy-Littlewood-Sobolev inequality of \cite{GaGa} and \eqref{heatkernel_k_2}, we obtain for $\Re s = -\frac{(n-1)}{2}$,
\begin{align*}
\|W^{-s}_1&(\Delta_g-z)^{s}W_2^{-s}\|_{\mc{C}_2(M)}^2 \\
&\leq C |\Gamma(-s)|^{-1}  \int_{M \times M} W_1(m)^{n-1} d(m,m')^{-2} W_2(m')^{n-1} dV_g(m) dV_g(m') \\
&\leq C |\Gamma(-s)|^{-1} \| W_1^{n-1} \|_{L^{\frac{n}{n-1}}(M)} \| W_2^{n-1} \|_{L^{\frac{n}{n-1}}(M)} \leq  C e^{C|\mathrm{Im} \,  s|} \| W_1\|_{L^n(M)}^{n-1} \| W_2\|_{L^n(M)}^{n-1}
\end{align*}
where the factor $e^{C|\mathrm{Im} \, s|}$ is contributed by the Gamma function. This shows \eqref{eq_negativez_2}.

We now interpolate using the family $S_s$ between  \eqref{eq_negativez_1} and  \eqref{eq_negativez_2}, as in the proof of Proposition~\ref{prop:Frank-Sabin}, and we obtain at $s = -1$
\begin{equation}
\| W_1 (\Delta_g - z)^{-1} W_2 \|_{\mathcal{C}_{n-1}} \leq C \| W_1\|_{L^n(M)} \| W_2\|_{L^n(M)}.
\end{equation}
which completes the proof for $W_1$ and $W_2$ non-negative and simple. The extension to general $W_1, W_2 \in L^n(M)$ is standard. 
\end{proof}

We now prove an analogue of Proposition~\ref{prop:resolvent-near-spectrum-away-from-singularity1}.
\begin{prop} 
\label{prop:resolvent-near-spectrum-away-from-singularity1_n/2}
Let $p = \frac{n}{2}$ and suppose $W_1, W_2\in L^n(M)$, and let $\phi$ be as in Proposition~\ref{prop:resolvent-near-spectrum-away-from-singularity1}. Then for $z > 0$, the operator
$W_1 \big(1-\phi\big)\big( \frac{\Delta_g}{z}\big) (\Delta_g - z)^{-1} W_2$
is in the Schatten class $\mathcal{C}_{n-1}(L^2(M))$ and  
\[
\bigg\|W_1 \bigg(1-\phi\bigg)\bigg( \frac{\Delta_g}{z}\bigg) (\Delta_g - z)^{-1} W_2 \bigg\|_{\mathcal{C}_{n-1}}\le C \| W_1\|_{L^n(M)} \| W_2\|_{L^n(M)},
\]
 uniformly in $z$. 
\end{prop}

\begin{proof}
We first note that for $z > 0$,  the operator
$$
W_1 \phi\bigg( \frac{\Delta_g}{z}\bigg)  (\Delta_g + z)^{-1} W_2 
$$
is in the Schatten class $\mathcal{C}_{n-1}(L^2(M))$, and 
\[
\bigg\|W_1 \phi\bigg( \frac{\Delta_g}{z}\bigg)  (\Delta_g + z)^{-1} W_2 \bigg\|_{\mathcal{C}_{n-1}}\le C\| W_1\|_{L^n(M)} \| W_2\|_{L^n(M)},
\]
uniformly in $z$. This follows from the spectral measure estimate \eqref{spec-meas-bound}, since 
$$
\int_0^\infty \lambda \phi \bigg( \frac{\lambda^2}{z}\bigg) (\lambda^2 + z)^{-1}  d\lambda
$$
is bounded uniformly in $z$. Combining this with the result of Proposition~\ref{negativez}, we see that $W_1 \big(1-\phi\big)\big(\frac{\Delta_g}{z}\big)(\Delta_g + z)^{-1} W_2 $ is in $\mathcal{C}_{n-1}(L^2(M))$ and we have 
\begin{equation}
\label{eq_compl_est_0}
\bigg\| W_1 \bigg(1-\phi\bigg)\bigg(\frac{\Delta_g}{z}\bigg)(\Delta_g + z)^{-1} W_2\bigg\|_{\mathcal{C}_{n-1}}\le C\| W_1\|_{L^n(M)} \| W_2\|_{L^n(M)},
\end{equation}
uniformly in $z$.

Now we write
\begin{equation}
\label{eq_compl_est_1}
\begin{aligned}
W_1 \bigg(1-\phi\bigg)\bigg(\frac{\Delta_g}{z}\bigg)(\Delta_g - z)^{-1} W_2  = W_1 \bigg(1-\phi\bigg)\bigg(\frac{\Delta_g}{z}\bigg)(\Delta_g + z)^{-1} W_2  \\ + 2zW_1 \bigg(1-\phi\bigg)\bigg(\frac{\Delta_g}{z}\bigg)(\Delta_g + z)^{-1}(\Delta_g-z)^{-1} W_2 .
\end{aligned}
\end{equation}
The first term in the right hand side of \eqref{eq_compl_est_1} has already been shown to lie in $\mathcal{C}_{n-1}$ with the bound \eqref{eq_compl_est_0}.  We write the second term on  the right hand side of \eqref{eq_compl_est_1} in terms of the spectral measure and apply Minkowski's integral inequality together with the spectral measure estimate \eqref{spec-meas-bound}, and find that the norm in $\mathcal{C}_{n-1}$ is bounded by 
$$
C\bigg(z  \int_0^\infty \bigg(1-\phi\bigg)\bigg( \frac{\lambda^2}{z}\bigg) (\lambda^2 + z)^{-1} (\lambda^2 - z)^{-1}  \lambda d\lambda\bigg) \| W_1\|_{L^n(M)} \| W_2\|_{L^n(M)}
$$
and a change of variable shows that this integral is convergent and independent of $z$, completing the proof.
\end{proof}

\section{Resolvent estimates on the spectrum. Completion of the proof of Theorem~\ref{thm_resolvent_Schatten_laplacian}}
\label{sec_on_spectrum}

The key difficulty in proving Theorem~\ref{thm_resolvent_Schatten_laplacian} is to obtain estimates on the limiting resolvent at the spectrum, $(\Delta_g - (z + i0))^{-1}$, for $z > 0$. Given Proposition~\ref{prop:resolvent-near-spectrum-away-from-singularity1} and Proposition \ref{prop:resolvent-near-spectrum-away-from-singularity1_n/2}, we only need to do this localized near the singularity at $z$ of the spectral multiplier $(\lambda^2 - z)^{-1}$. In doing so, following \cite{GuHa}, \cite{GHS},  and \cite{HaZh}, we shall use a microlocal partition of unity.

\subsection{Operator partition of unity}

\label{sec:partition of unity}
We begin by recalling some results of \cite{GuHa} and \cite{HaZh} on high and low frequency microlocal estimates on the spectral measure and resolvents of $\Delta_g$. 

\begin{prop}\label{prop:partition of unity}
\emph{\textbf{High frequency microlocal estimates.}}  For all high energies $\eta \geq 1/2$,  there exists a family of bounded operators $Q_i(\eta):L^2(M)\to L^2(M)$, $i=1,\dots,N_h$, with $N_h$ independent of $\eta$ and with the norm satisfying
\begin{equation}
\|Q_i(\eta)\|_{L^2(M)\to L^2(M)}\leq C \text{ for some $C$ independent of $\eta$},
\label{norm-uniform}
\end{equation}
 so that the following properties hold:\\
(1) The operators $Q_i(\eta)$ form an operator partition of unity: 
\begin{equation}
\sum_{i=1}^{N_h}Q_i(\eta)={\rm Id}.
\label{partition-of-identity}\end{equation}
(2)  Let  $\eta\geq 1/2$ and $(i,j) \in \{1,\dots, N_h\}^2$. There exists $\delta>0$ small such that for all $z>0$ such that $\sqrt{z} \in  [(1-\delta)\eta, (1+\delta)\eta]$,  one of the following  three alternatives holds:\\
(2.i) One has, for the \emph{outgoing} resolvent, 
 \begin{equation}\label{1stcase}
\big(Q_i(\eta)^*(\Delta_g-(z + i0))^{-1}Q_j(\eta)\big)(m,m')\in 
x(m)^\infty {x(m')}^\infty z^{-\infty}C^\infty(\bbar{M}\x\bbar{M}),
\end{equation} 
for all $m,m'\in M$, where the $C^\infty(\bbar{M}\x\bbar{M})$ part depends also on $z$ and is uniformly bounded in $z$ in the smooth topology. \\
(2.ii) One has for the \emph{incoming} resolvent,
\begin{equation}
\big(Q_i(\eta)^*(\Delta_g-(z -i0))^{-1}Q_j(\eta)\big)(m,m')\in 
x(m)^\infty {x(m')}^\infty z^{-\infty}C^\infty(\bbar{M}\x\bbar{M}),
\label{2ndcase}\end{equation}
for all $m,m'\in M$.

(2.iii) The spectral measure 
satisfies, for $\lambda=\sqrt{z} \in  [(1-\delta)\eta, (1+\delta)\eta]$, the following bounds: 
for all $k=0,1,2,\dots$, there is $C_k>0$ such that for all $m,m'\in M$ 
\begin{equation}
\label{eq_10_1_partition} 
\Big|\pl_{\la}^k\big(Q_i(\eta)^*dE_{\sqrt{\Delta_g}}(\la)Q_j(\eta)\big)(m,m')\Big|\leq C_k\la^{n-1-k}(1+\la d(m,m'))^{-\frac{(n-1)}{2}+k},
\end{equation}
\begin{equation}
\label{eq_10_2_partition}
\big(Q_i(\eta)^*dE_{\sqrt{\Delta_g}}(\la)Q_j(\eta)\big)(m,m')=\la^{n-1}\bigg(\sum_{\pm}e^{\pm i\la d(m,m')}a_\pm(\la,m,m')+b(\la,m,m')\bigg),
\end{equation}
with $a_\pm,b$ satisfying the estimates for all $k=0,1,2,\dots$, 
\begin{equation}
\label{eq_10_3_partition}
 |\pl_\la^k a_\pm(\la,m,m')|\leq C_k\la^{-k}(1+\la d(m,m'))^{-\frac{(n-1)}{2}}, 
\end{equation}
 \begin{equation}
\label{eq_10_4_partition}
|\pl_\la^k b(\la,m,m')|\leq C_k\la^{-k}(1+\la d(m,m'))^{-K}, \quad \forall K>1.
\end{equation}
Moreover the alternative (2.iii) always holds if $i=j$.

\emph{\textbf{Low frequency microlocal estimates.}}  
Similarly, for all low energies $\eta \leq 2$, there exists a family of bounded operators $Q_i(\eta):L^2(M)\to L^2(M)$, $i=0, *, 1,\dots,N_l$, with $N_l$ independent of $\eta$ satisfying \eqref{norm-uniform} and \eqref{partition-of-identity} (with the sum in this case ranging over $i=0, *, 1,\dots,N_l$), 
satisfying the following:

(3)  Let $0<\eta\leq 2$ and $i,j$ range independently in $\{0,*, 1, \dots, N_l\}$. There exists $\delta>0$ small such that for all $z>0$ satisfying $\lambda := \sqrt{z} \in [(1-\delta)\eta, (1+\delta)\eta]$, one of the following three alternatives holds:\\
(3.i) One has the pointwise kernel bound for the \emph{outgoing} resolvent (for all $N\in\NN$) 
\begin{equation}
\label{1stcaselow}
\begin{split}
|\big(Q_i(\eta)^*(\Delta_g-(z+i0))^{-1} & Q_j(\eta)\big)(m,m')|\leq \\
& C_N \Big( \frac{x}{x+\lambda}\Big)^N \Big( \frac{x'}{x'+\lambda}\Big)^N \frac{(xx')^{\frac{n-1}{2}}(\chi(\tfrac{x}{\la})+\chi(\tfrac{x'}{\la}))}{x+x'+\la}, 
\end{split}\end{equation} 
where $x=x(m), x'=x(m')$, and $\chi\in C_0^\infty((-\varepsilon,\varepsilon),[0,\infty))$  is such that $\chi=1$ in $[-\varepsilon/2,\varepsilon/2]$.  Here $\varepsilon>0$ is small enough. \\
(3.ii) One has the pointwise kernel bound for the \emph{incoming} resolvent (for all $N\in\NN$)
\begin{equation}
\label{2ndcaselow}
\begin{split}
|\big(Q_i(\eta)^*(\Delta_g-(z-i0))^{-1} & Q_j(\eta)\big)(m,m')|
\leq \\
& C_N \Big( \frac{x}{x+\lambda}\Big)^N \Big( \frac{x'}{x'+\lambda}\Big)^N  \frac{(xx')^{\frac{n-1}{2}}(\chi(\tfrac{x}{\la})+\chi(\tfrac{x'}{\la}))}{x+x'+\la}. 
\end{split}
\end{equation}\\
(3.iii)  For all $k=0,1,2,\dots$, there is $C_k>0$ such that \eqref{eq_10_1_partition}, \eqref{eq_10_2_partition}, \eqref{eq_10_3_partition} and \eqref{eq_10_4_partition} hold. 

Moreover if $i=j$, the alternative (3.iii) holds.\\
\end{prop}

\begin{rem} The two partitions of the identity do not quite match up in the intermediate energy regime, $1/2 \leq \eta \leq 2$. Because of this, it would be more notationally accurate to label the partitions $Q_i^{high}$ and $Q_j^{low}$; to avoid cumbersome notation, we do not do this. We emphasize that in this intermediate regime, either partition can be used. 
\end{rem}

\begin{rem}
In the low energy case, $\eta \leq 2$, let us first point out the meaning of the RHS of \eqref{1stcaselow} and \eqref{2ndcaselow}. In \cite{GuHaSi_III} it was shown that the Schwartz kernel of the resolvent $(\Delta_g - (\lambda^2 \pm i0))^{-1}$ for $\lambda \in [0, \lambda_0]$ has some polyhomogeneous structure on the ``low energy space'', which is a blowup of $\bbar{M} \times \bbar{M} \times [0, \lambda_0]$. Ignoring the artificial boundary at $\lambda = \lambda_0$, this bown-up space has 7 boundary hypersurfaces corresponding to 7 different types of asymptotics displayed by the resolvent kernel. These are the left boundary $\lb$, the right boundary $\rb$, which arise from $\pl \bbar{M} \times \bbar{M} \times [0, \lambda_0]$ and $\bbar{M} \times \pl \bbar{M} \times [0, \lambda_0]$; the b-face $\bfc$, which arises from blowing up $\pl \bbar{M} \times \pl \bbar{M} \times [0, \lambda_0]$; the `zero face' $\zf$, arising from $\bbar{M} \times \bbar{M} \times \{ 0 \}$; and three faces at $\lambda=0$ produced by blowing up. These are $\bfo$, arising from blowing up $\pl \bbar{M} \times \pl \bbar{M} \times \{ 0 \}$; the face $\lbo$, arising from blowing up $\pl \bbar{M} \times  \bbar{M} \times \{ 0 \}$; and lastly $\rbo$, arising from blowing up $ \bbar{M} \times \pl \bbar{M} \times \{ 0 \}$. See
 figure 1 of \cite{GuHaSi_III}. 

The resolvent (microlocally away from the conormal bundle of the diagonal) was shown in \cite{GuHaSi_III} to be polyhomogeneous and vanish to order $n-2$ at the boundary hypersurfaces $\lbo$, $\rbo$, $\bfo$, and to vanish to order $(n-1)/2$ at $\lb$ and $\rb$. 
Cases (3.i) and (3.ii) apply when the microlocalizing operators $Q_i$ and $Q_j$ remove the wavefront set at $\lb, \rb$ and $\bfc$, meaning there is infinite order vanishing there. Moreover, the cutoff factor $\chi(x/\lambda) +\chi(x'/\lambda)$  vanish in a neighbourhood of $\zf$. Now notice that  
$x$ vanishes to first order at $\lb$, $\lbo$ and $\bfo$, while $x'$ vanishes to first order at $\rb$, $\rbo$ and $\bfo$ and $x+x'+\lambda$ vanishes to first order at $\bfo$. So the product on the RHS of \eqref{1stcaselow} and \eqref{2ndcaselow} precisely encodes the order of vanishing at these remaining boundary hypersurfaces. 
\end{rem}

\begin{proof} This is a combination of several results from  \cite{GuHa} and \cite{GHS}. In the high energy case, $\eta \geq 1/2$, Lemma 5.3 of \cite{GuHa} tells us that the pairs $(i,j)$ split into four cases. In the first two cases, $Q_i(\eta)^*$ is either not-incoming or not-outgoing related to $Q_j(\eta)$, and then Proposition 6.7 of \cite{GuHa} applies; note that the estimates in (2.i) and (2.ii) above appear in the proof, rather than the statement, of Proposition 6.7. In the third and fourth cases, Theorem 1.12 of \cite{GHS} applies and shows that estimates \eqref{eq_10_1_partition} hold, see also Proposition 6.4 of \cite{GuHa}. 
Also in the third and fourth cases,  Proposition 1.5 of \cite{HaZh} holds and gives the estimates \eqref{eq_10_2_partition}, \eqref{eq_10_3_partition} and \eqref{eq_10_4_partition}. Note that 
 \cite[Proposition 1.5]{HaZh} is written in the case when $i=j$ but the proof of that proposition shows that it remains valid more generally when $i\ne j$ but the microsupports are close enough.

In the low energy case, as shown in Section 6 of \cite{GuHa}, case (3.iii) applies to the pairs $(0,0)$, $(*, *)$, and $(i, j)$ where $i, j\geq 1$ and $|i-j| \leq 1$. Moreover, case (3.iii) also applies to any pair where either $i = *$ or $j=*$. That is because in these cases, the operator $Q_*(\eta)$ annihilates all the wavefront set of the spectral measure at bf, with the consequence that the spectral measure estimates 
\begin{equation}
\label{eq_10_1} 
\Big|\pl_{\la}^k\big(Q_i(\eta)^*dE_{\sqrt{\Delta_g}}(\la)Q_j(\eta)\big)(m,m')\Big|\leq C_k\la^{n-1-k}(1+\la d(m,m'))^{-\frac{(n-1)}{2}+k},
\end{equation}
 hold if either $i = *$ or $j=*$, and this leads to estimates \eqref{eq_10_1_partition} as in the high energy case.  For (3.iii) with $i,j\geq 1$, the estimates \eqref{eq_10_2_partition}, \eqref{eq_10_3_partition} and \eqref{eq_10_4_partition} are proven in 
 \cite[Proposition 1.5]{HaZh} in the case when $i=j$ but the proof shows that it remains valid more generally when $i\ne j$ but the microsupports are close enough.  
 The case $i,j\in\{0,*\}$ in (3.iii) is also shown in \cite[Proposition 1.5]{HaZh}.

 The cases $i=0$ and $j \geq 1$, or $i\geq 1$ and $j=0$, fit any one of the cases (3.i), (3.ii), (3.iii) above. This is because here the wavefront set at bf is wiped out by $Q_0(\eta)$, while the wavefront set at fibre-infinity is wiped out by $Q_j(\eta)$ for $j \geq 1$. 
 
The final case remaining, where $i, j \geq 1$ and $|i-j| \geq 2$, fit into cases (3.i) or (3.ii) according to whether $Q_i(\eta)^*$ is not incoming-related or not outgoing-related to $Q_j(\eta)$, as shown in Proposition 6.9 of \cite{GuHa}. 
\end{proof}

Cases (3.i) and (3.ii) will be treated using the following lemma. 

\begin{lem}
\label{lem:K}
Let $(M,g)$ be an asymptotically conic manifold of dimension $n\ge 3$. Then if an integral operator $K$ has kernel $K(m,m')$ bounded pointwise by 
$$
C\frac{(xx')^{\frac{n-1}{2}}(\chi(\tfrac{x}{\la})+\chi(\tfrac{x'}{\la}))}{x+x'+\la},\quad 0<\lambda\le 3, 
$$
then for $W_1, W_2 \in L^{2p}(M)$, $p \in [\frac{n}{2}, \frac{n+1}{2}]$, the operator $W_1 K W_2$ is Hilbert Schmidt and we have 
\begin{equation}
\label{eq_lem_K}
\| W_1 K W_2 \|_{\mathcal{C}_2}\le C \lambda^{-2+\frac{n}{p}} \| W_1 \|_{L^{2p}(M)}\| W_2 \|_{L^{2p}(M)}.
\end{equation}
\end{lem}

\begin{proof} 

Using H\"older's inequality with $1/p'+1/p=1$ and $p'\in [\frac{n+1}{n-1},\frac{n}{n-2}]$, we get 
\begin{align*} 
\|W_1K W_2\|_{\mc{C}_2}& \le 
\|W_1\|_{L^{2p}} \|W_2\|_{L^{2p}} \\
& \Big(\int_{M\x M}\frac{(x(m)x(m'))^{(n-1)p'}(\chi(\tfrac{x(m)}{\la})+\chi(\tfrac{x(m')}{\la}))^{2p'}}{(x(m)+x(m')+\la)^{2p'}}dV_g(m)dV_g(m')\Big)^{1/2p'}.
\end{align*}
We use the coordinates $m=(x,y), m'=(x',y')$ near the boundary, where the measure $dV_g(m)$ is comparable to 
$\frac{dxdy}{x^{n+1}}$.  Let us introduce the polar coordinates $(x,x')=(R\sin(\theta),R\cos(\theta))$ with $\theta\in[0,\pi/2]$, near $x=x'=0$. Using that $(n-1)p'-(n+1)\ge 0$ and $x+x'\sim R$, we get 
\begin{align*}
\Big(&\int_{M\x M}\frac{(xx')^{(n-1)p'}\chi(\tfrac{x}{\la})}{(x+x'+\la)^{2p'}}dV_{g}dV_{g'}\Big)^{\frac{1}{2p'}}
\le C \Big( \int_{0<x<2\lambda} \frac{(xx')^{(n-1)p'-(n+1)}}{(x+x'+\la)^{2p'}}dxdx' \Big)^{\frac{1}{2p'}}\\
&\le C \Big( \int_0^\infty \int_{0<\sin\theta<2\lambda/R}  \frac{R^{2(n-1)p'-2n-1}}{(R+\lambda)^{2p'}}dRd\theta  \Big)^{\frac{1}{2p'}}\\
&\le C\frac{1}{\lambda} \Big( \int_0^{2\lambda}  R^{2(n-1)p'-2n-1}  dR\Big)^{\frac{1}{2p'}}
+ C \Big( \int_{2\lambda}^\infty \int_{0<\theta\le \tilde C\lambda/R} R^{2(n-1)p'-2p'-2n-1}  dRd\theta\Big)^{\frac{1}{2p'}} \\
&\le C\lambda^{\frac{n}{p}-2}+ C\lambda^{\frac{1}{2p'}}  \Big( \int_{2\lambda}^\infty  R^{2(n-2)p'-2n-2}  dR\Big)^{\frac{1}{2p'}} \le C\lambda^{\frac{n}{p}-2}.
\end{align*}
Here we used that $(n-1)p'>n$ and $2(n-2)p'-2n-1<0$. The same argument works  with the term involving $\chi(x'/\la)$ and the estimate \eqref{eq_lem_K} follows. 
\end{proof}

\subsection{Analytic family of operators}   In this section we closely follow Section 4 of \cite{GuHa}, especially Remark 4.2 (which is essentially due to  Adam Sikora).  Let $\phi\in C^\infty_0(((1-\delta/4)^2,(1+\delta/4)^2))$ be such that $\phi(t)=1$ in a neighborhood of $t=1$, where $\delta>0$ is small, and consider the analytic family of  operators  in  $\Re(s)\le 0$,
\[
H_{s,z,\varepsilon}(\Delta_g)=\phi\bigg(\frac{\Delta_g}{z}\bigg)  (\Delta_g-(z+i\varepsilon))^s, \quad z>0,\quad \varepsilon>0.
\] 
By the spectral theorem, we have 
\begin{equation}
\label{eq_def_H_s_z}
H_{s,z,\varepsilon}(\Delta_g)=z^{s+\frac{1}{2}}\int_0^\infty \bigg(\lambda-\bigg(1+i\frac{\varepsilon}{z}\bigg)\bigg)^s\frac{\phi(\lambda)}{2\sqrt{\lambda}} dE_{\sqrt{\Delta_g}}(z^{\frac{1}{2}}\lambda^{\frac{1}{2}})d\lambda.
\end{equation}
Let $\eta>0$ be such that $z^{1/2}\in  [(1-\delta/2)\eta, (1+\delta/2)\eta]$  and let  $Q_i(\eta)$ and $Q_j(\eta)$ be such that the condition (2.iii) or (3.iii) of Proposition \ref{prop:partition of unity} holds, in the high energy, respectively, low energy case. Then using \eqref{eq_def_H_s_z}, we have on the level of Schwartz kernels, for $m,m'\in M$,
\begin{equation}
\label{eq_def_H_s_z_with_Q}
\big(Q_i(\eta)^*H_{s,z,\varepsilon}(\Delta_g)Q_j(\eta)\big)(m,m')=z^{s+\frac{1}{2}}\int_0^\infty \bigg(\lambda-\bigg(1+i\frac{\varepsilon}{z}\bigg)\bigg)^s\psi(\lambda)d\lambda,
\end{equation}
where 
\[
\psi(\lambda)=\frac{\phi(\lambda)}{2\sqrt{\lambda}} Q_i(\eta)^* dE_{\sqrt{\Delta_g}}(z^{\frac{1}{2}}\lambda^{\frac{1}{2}}) Q_j(\eta)(m,m').
\]
Here, as $\delta>0$ is small, we have $z^{1/2}\lambda^{1/2}\in [(1-\delta)\eta,(1+\delta)\eta]$ when $z^{1/2}\in  [(1-\delta/2)\eta, (1+\delta/2)\eta]$ and $\lambda\in \supp(\phi)$, and therefore, in view of \eqref{eq_10_1_partition}, we have
$\psi(\lambda) \in C^\infty_0(\R)$. 

Letting $\varepsilon\to 0$ in \eqref{eq_def_H_s_z_with_Q}, we define the operators $Q_i(\eta)^*H_{s,z,0}(\Delta_g)Q_j(\eta)$, when $z^{1/2}\in  [(1-\delta/2)\eta, (1+\delta/2)\eta]$,  as operators whose Schwartz kernels are given by 
\begin{equation}
\label{eq_def_H_s_z_with_Q_eps=0}
\begin{aligned}
\big(Q_i(\eta)^* H_{s,z,0}(\Delta_g)Q_j(\eta)\big)(m,m')&=z^{s+\frac{1}{2}}\int_0^\infty (\lambda-(1+i0))^s  \psi(\lambda) d\lambda\\
&= z^{s+\frac{1}{2}} \bigg(  (\lambda -i0)^s *  \psi(\lambda) \bigg)(1).
\end{aligned}
\end{equation}

We are interested in pointwise estimates for the  kernel of $Q_i(\eta)^*H_{s,z,0}(\Delta_g)Q_j(\eta)$ and to this end, we shall need the following result of \cite[Remark 4.2]{GuHa}. Even though the proof is almost the same as that of \cite[Lemma 3.3]{GHS}, for completeness we provide a proof in the Appendix \ref{app}. 
\begin{lem}
\label{Guillarmou_Hassell_remark}
Let $a<b<c\le 0$ and let us  write $b=\theta a +(1-\theta) c$, $0< \theta< 1$. Then there is $C>0$ such that for all $f\in C^\infty_0(\R)$, all $t\in \R$, and all $0<\varepsilon\ll 1$,  we have 
\begin{equation}
\label{eq_10_4_conv}
\|(\lambda\pm i \varepsilon)^{b+it}* f\|_{L^\infty_\lambda}\le C(1+|t|)e^{\frac{3\pi|t|}{2}}\|\chi_+^a*f\|_{L^\infty_\lambda}^\theta \|\chi_+^c*f\|_{L^\infty_\lambda}^{1-\theta}. 
\end{equation}
\end{lem}

We have the following result. 

\begin{prop}
\label{prop:Hsz} Suppose that $(i,j)$ are such that the condition (2.iii) or (3.iii) holds, in the high energy, respectively,  low energy case. Then there is $C>0$ such that the kernel of the operator $Q_i(\eta)^* H_{s,z,0}(\Delta_g)Q_j(\eta)$ with $z>0$ and $z^{\frac{1}{2}}\in [(1-\delta/2)\eta,(1+\delta/2)\eta]$ has the following pointwise estimates, 

(i) For $\emph{\text{Re}}(s) = -\frac{(n+1)}{2}$, we have 
\begin{equation}
\Big| Q_i(\eta)^* H_{s,z,0}(\Delta_g)Q_j(\eta)(m,m') \Big| \leq C e^{C|\emph{\text{Im}}\,(s)|}z^{-\frac{1}{2}} 
\label{Hs=-(n+1)/2}\end{equation}
for all $m,m'\in M$, uniformly in $z$ and $\eta$.

(ii) For $\emph{\text{Re}} (s) = -\frac{(n-1)}{2}$, we have 
\begin{equation}
\Big| Q_i(\eta)^* H_{s,z,0}(\Delta_g)Q_j(\eta)(m,m')  \Big| \leq C e^{C|\emph{\text{Im}}\,(s)|} d(m,m')^{-1}. 
\label{Hs=-(n-1)/2}
\end{equation}
for all $m,m'\in M$, uniformly in $z$ and $\eta$.
\end{prop}

\begin{proof} Estimate \eqref{Hs=-(n+1)/2} is proved in \cite[Remark 4.2]{GuHa}. Estimate \eqref{Hs=-(n-1)/2} is proved in the same way, except for the case $n=3$, relying on the estimates \eqref{eq_10_1_partition} only. Indeed, in the case $n\ge 5$ is odd, we take 
$a = -\frac{(n+1)}{2}$ and  $c = -\frac{(n-3)}{2}$
 in Lemma~\ref{Guillarmou_Hassell_remark} and using that 
 \[
\chi_+^{-k}=\delta_0^{(k-1)}, \quad k=1,2,\dots, 
\] 
we get 
\begin{align*}
\Big| Q_i(\eta)^* H_{s,z,0}(\Delta_g)Q_j(\eta)(m,m')  \Big| &\leq Cz^{\frac{2-n}{2}}(1+|\text{Im}(s)|)e^{\frac{3\pi |\text{Im} (s)|}{2}}\\
& \x\bigg\|\p_\lambda^{\frac{n-1}{2}}\bigg(\frac{\phi(\lambda)}{2\sqrt{\lambda}} Q_i(\eta)^* dE_{\sqrt{\Delta_g}}(z^{\frac{1}{2}}\lambda^{\frac{1}{2}}) Q_j(\eta)(m,m')\bigg)\bigg\|_{L^\infty}^{1/2} \\
&\x \bigg\|\p_\lambda^{\frac{n-5}{2}}\bigg(\frac{\phi(\lambda)}{2\sqrt{\lambda}} Q_i(\eta)^* dE_{\sqrt{\Delta_g}}(z^{\frac{1}{2}}\lambda^{\frac{1}{2}}) Q_j(\eta)(m,m')\bigg)\bigg\|_{L^\infty}^{1/2}, 
\end{align*}
and therefore, using  \eqref{eq_10_1_partition}, we obtain that 
\begin{equation}
\label{eq_a_b_odd}
\begin{aligned}
\Big| Q_i(\eta)^* H_{s,z,0}(\Delta_g)Q_j(\eta)(m,m')  \Big|&\le Ce^{C|\text{Im}(s)|}z^{\frac{1}{2}}(1+z^{\frac{1}{2}}d(m,m'))^{-1}\\
&\le C e^{C|\text{Im}\,(s)|} d(m,m')^{-1}.
\end{aligned}
\end{equation}
For $n \geq 4$ even, taking  $a = -\frac{n}{2}$, $c = -\frac{(n-2)}{2}$ in Lemma~\ref{Guillarmou_Hassell_remark} and using \eqref{eq_10_1_partition}, we also get \eqref{eq_a_b_odd}. 
We have therefore established \eqref{Hs=-(n-1)/2} for all $n\ge 4$.

When $n=3$, using Lemma~\ref{Guillarmou_Hassell_remark} with $a=-2$ and $c=0$, and the fact that $\chi_+^0(\lambda)=H(\lambda)$ is the Heaviside function, we obtain that 
\begin{equation}
\label{eq_n=3_pointwise_1}
\begin{aligned}
\Big| Q_i(\eta)^* H_{s,z,0}(\Delta_g)Q_j(\eta)(m,m')  \Big| &\le Cz^{-\frac{1}{2}}(1+|\text{Im}(s)|)e^{\frac{3\pi |\text{Im} (s)|}{2}}\\
&\x \bigg\|\p_\lambda \bigg(\frac{\phi(\lambda)}{2\sqrt{\lambda}} Q_i(\eta)^* dE_{\sqrt{\Delta_g}}(z^{\frac{1}{2}}\lambda^{\frac{1}{2}}) Q_j(\eta)(m,m')\bigg)\bigg\|_{L^\infty}^{1/2} \\
&\x\bigg\|H*\bigg(\frac{\phi(\lambda)}{2\sqrt{\lambda}} Q_i(\eta)^* dE_{\sqrt{\Delta_g}}(z^{\frac{1}{2}}\lambda^{\frac{1}{2}}) Q_j(\eta)(m,m')\bigg)\bigg\|_{L^\infty}^{1/2}.
\end{aligned}
\end{equation} 
By \eqref{eq_10_1_partition}, we get 
\begin{equation}
\label{eq_n=3_pointwise_2}
\bigg\|\p_\lambda \bigg(\frac{\phi(\lambda)}{2\sqrt{\lambda}} Q_i(\eta)^* dE_{\sqrt{\Delta_g}}(z^{\frac{1}{2}}\lambda^{\frac{1}{2}}) Q_j(\eta)(m,m')\bigg)\bigg\|_{L^\infty}\le Cz.
\end{equation}
Now if we show that 
\begin{equation}
\label{eq_n=3_pointwise_3}
\bigg\|H*\bigg(\frac{\phi(\lambda)}{2\sqrt{\lambda}} Q_i(\eta)^* dE_{\sqrt{\Delta_g}}(z^{\frac{1}{2}}\lambda^{\frac{1}{2}}) Q_j(\eta)(m,m')\bigg)\bigg\|_{L^\infty}\le Cd(m,m')^{-2},
\end{equation}
then the estimate   \eqref{Hs=-(n-1)/2} will follow from \eqref{eq_n=3_pointwise_1}, \eqref{eq_n=3_pointwise_2} and \eqref{eq_n=3_pointwise_3}. To prove \eqref{eq_n=3_pointwise_3}, using \eqref{eq_10_2_partition}, we write 
\begin{equation}
\label{eq_n=3_pointwise_4}
\begin{aligned}
H*\bigg(\frac{\phi(\lambda)}{2\sqrt{\lambda}} &Q_i(\eta)^* dE_{\sqrt{\Delta_g}}(z^{\frac{1}{2}}\lambda^{\frac{1}{2}}) Q_j(\eta)(m,m')\bigg)(\lambda)\\
&=\int_{0}^{\lambda^{\frac{1}{2}}} \phi(\mu^2) Q_i(\eta)^* dE_{\sqrt{\Delta_g}}(z^{\frac{1}{2}}\mu ) Q_j(\eta)(m,m')d\mu\\
&=\int_{0}^{\lambda^{\frac{1}{2}}} \phi(\mu^2) z\mu^2 \bigg[\sum_{\pm} e^{\pm i z^{\frac{1}{2}}\mu d(m,m')}a_\pm(z^{\frac{1}{2}}\mu,m,m')+b(z^{\frac{1}{2}}\mu,m,m')\bigg] d\mu.
\end{aligned}
\end{equation}
The terms involving $a_\pm$ in  \eqref{eq_n=3_pointwise_4} can be treated similarly and in what follows we shall only consider the term involving $a_+$ and drop the sign $+$.  To estimate this term, we integrate by parts and get 
\begin{equation}
\label{eq_n=3_pointwise_6}
\begin{aligned}
\int_{0}^{\lambda^{\frac{1}{2}}} \phi(\mu^2) z\mu^2 e^{i z^{\frac{1}{2}}\mu d(m,m')}&a(z^{\frac{1}{2}}\mu,m,m')d\mu\\
&=\frac{1}{i z^{\frac{1}{2}}d(m,m')}\bigg[   \phi(\mu^2) z\mu^2 e^{i z^{\frac{1}{2}}\mu d(m,m')}a(z^{\frac{1}{2}}\mu,m,m')|_{\mu=0}^{\mu=\lambda^{\frac{1}{2}}} \\
&-\int_{0}^{\lambda^{\frac{1}{2}}}\p_{\mu} \big( \phi(\mu^2) z\mu^2 a(z^{\frac{1}{2}}\mu,m,m') \big) e^{i z^{\frac{1}{2}}\mu d(m,m')} d\mu\bigg].
\end{aligned}
\end{equation} 
Estimating the terms in the left hand side of \eqref{eq_n=3_pointwise_6} with the help of \eqref{eq_10_3_partition}, we obtain that  
\begin{equation}
\label{eq_n=3_pointwise_7}
\bigg| \int_{0}^{\lambda^{\frac{1}{2}}} \phi(\mu^2) z\mu^2 e^{i z^{\frac{1}{2}}\mu d(m,m')}a(z^{\frac{1}{2}}\mu,m,m')d\mu\bigg|\le C\lambda^{\frac{1}{2}}d(m,m')^{-2},
\end{equation}
 uniformly in $z$.  To estimate the term involving the remainder $b$ in  \eqref{eq_n=3_pointwise_4}, we use \eqref{eq_10_4_partition} with $K=2$ and get
\begin{equation}
\label{eq_n=3_pointwise_5}
\begin{aligned}
\int_{0}^{\lambda^{\frac{1}{2}}} \phi(\mu^2) z\mu^2|b(z^{\frac{1}{2}}\mu,m,m')|d\mu&\le C\int_{0}^{\lambda^{\frac{1}{2}}} \phi(\mu^2) z\mu^2 (1+z^{\frac{1}{2}}\mu d(m,m'))^{-2}d\mu\\
&\le Cd(m,m')^{-2}.
\end{aligned}
\end{equation}
 Now \eqref{eq_n=3_pointwise_3} follows from \eqref{eq_n=3_pointwise_4},   \eqref{eq_n=3_pointwise_7} and \eqref{eq_n=3_pointwise_5}. This completes the proof of estimate \eqref{Hs=-(n-1)/2}. 
\end{proof}

When proving the Schatten bound on the resolvent on the spectrum in Section \ref{sec_resolvent_on_spec} below, the cases (2.iii) and (3.iii) of   Proposition~\ref{prop:partition of unity} will be treated using the following result. 

\begin{prop} 
\label{prop_cases_2_iii_3_iii}
Suppose that $(i,j)$ are such that the condition (2.iii) or (3.iii) holds, in the high energy, respectively low energy case.  Let $p\in [\frac{n}{2},\frac{n+1}{2}]$. Then there is $C>0$ such that for all $z\in (0,\infty)$, $z^{\frac{1}{2}}\in [(1-\delta/2)\eta,(1+\delta/2)\eta]$, and all $W_1, W_2\in L^{2p}(M)$, we have $W_1Q_i(\eta)^*H_{-1,z,0}(\Delta_g)Q_j(\eta)W_2\in \mathcal{C}_q(L^2(M))$, $q=\frac{p(n-1)}{n-p}$, and 
\begin{equation}
\label{eq_prop_cases_2_iii_3_iii}
\|W_1Q_i(\eta)^*H_{-1,z,0}(\Delta_g)Q_j(\eta)W_2\|_{\mathcal{C}_q}\le C z^{-1+\frac{n}{2p}}\|W_1\|_{L^{2p}(M)}\|W_2\|_{L^{2p}(M)}.
\end{equation}
\end{prop}

\begin{proof}
First thanks to Proposition~\ref{prop:Hsz}, case (i), we know that for $\Re s=-\frac{(n+1)}{2}$, 
\[
\|Q_i(\eta)^* H_{s, z,0}(\Delta_g) Q_j(\eta)\|_{L^1(M)\to L^\infty(M)}\le Ce^{C|\text{Im}(s)|}z^{-\frac{1}{2}}.
\] 
By spectral theorem, we also know that for $\Re s=0$, 
\[
\|Q_i(\eta)^* H_{s, z,0}(\Delta_g) Q_j(\eta)\|_{L^2(M)\to L^2(M)}\le Ce^{\pi |\text{Im}(s)|}.
\] 
Hence, Proposition~\ref{prop:Frank-Sabin} implies that $W_1 Q_i(\eta)^* H_{-1, z,0}(\Delta_g) Q_j(\eta)W_2\in \mathcal{C}_{n+1}(L^2(M))$ and moreover, 
\begin{equation}
\label{eq_zreel_res_2}
\|W_1 Q_i(\eta)^* H_{-1, z,0}(\Delta_g) Q_j(\eta)W_2\|_{\mathcal{C}_{n+1}}\le C z^{-\frac{1}{n+1}}\|W_1\|_{L^{n+1}(M)}\|W_2\|_{L^{n+1}(M)}.
\end{equation}
Now when  $\Re s = -\frac{(n-1)}{2}$, thanks to Proposition~\ref{prop:Hsz} (ii),  the kernel of the operator $Q_i(\eta)^* H_{s, z,0}(\Delta_g) Q_j(\eta)$ has the bound \eqref{Hs=-(n-1)/2}, which is the same as the bound \eqref{heatkernel_k_2} in the proof of Proposition~\ref{negativez}. Proceeding exactly as in the proof of Proposition~\ref{negativez}, we get 
\begin{equation}
\label{eq_zreel_res_3}
\|W_1 Q_i(\eta)^* H_{-1, z,0}(\Delta_g) Q_j(\eta)W_2\|_{\mathcal{C}_{n-1}}\le C \|W_1\|_{L^{n}(M)}\|W_2\|_{L^{n}(M)}.
\end{equation}

In view of  \eqref{eq_zreel_res_2} and \eqref{eq_zreel_res_3}, 
the bound \eqref{eq_prop_cases_2_iii_3_iii} follows by a complex interpolation argument applied to the analytic family of operators 
\[
\zeta\mapsto W_1^{\frac{2}{n+1}+\zeta\frac{2}{n(n+1)}} Q_i(\eta)^* H_{-1, z,0}(\Delta_g) Q_j(\eta)  W_2^{\frac{2}{n+1}+\zeta\frac{2}{n(n+1)}} 
\]
 in the strip $0\le \Re\zeta\le 1$, with $W_j\ge 0$ being simple functions such that $\|W_j\|_{L^2(M)}=1$, $j=1,2$, see  \cite[p. 154]{Simon_operator_theory}.
\end{proof}

\subsection{Resolvent estimates on the spectrum}

\label{sec_resolvent_on_spec}

The final ingredient in the proof of  Theorem~\ref{thm_resolvent_Schatten_laplacian} is the following result. 
\begin{prop}
\label{zreel}
Let $\phi\in C^\infty_0(((1-\delta/4)^2,(1+\delta/4)^2))$ be such that $\phi(t)=1$ in a neighborhood of $t=1$, where $\delta>0$ is small, and let $p\in [\frac{n}{2},\frac{n+1}{2}]$. 
Then
there is $C>0$ such that for all $z\in (0,\infty)$ and all $W_1, W_2\in L^{2p}(M)$, then for $q=\frac{p(n-1)}{n-p}$ we have $W_1 \phi(\frac{\Delta_g}{z}) (\Delta_g-(z+i0))^{-1}W_2\in \mathcal{C}_q(L^2(M))$ and
\begin{equation}
\label{eq_zreel_res}
\bigg\|W_1   \phi\bigg(\frac{\Delta_g}{z}\bigg)(\Delta_g-(z+i0))^{-1}W_2\bigg\|_{\mc{C}_{q}}\leq Cz^{-1+\frac{n}{2p}}\|W_1\|_{L^{2p}(M)} \|W_2\|_{L^{2p}(M)}.
\end{equation}
\end{prop}

\begin{proof}
Let us first take the high energy case $z \geq 1$ and let $\eta\ge 1$ be such that  $\sqrt{z}\in [(1-\delta/2)\eta,(1+\delta/2)\eta]$.  We decompose the spectrally localized outgoing resolvent $\phi(\frac{\Delta_g}{z})(\Delta_g - (z+i0))^{-1}$ into microlocalized pieces
\[ 
W_1 \phi\bigg(\frac{\Delta_g}{z}\bigg) (\Delta_g-(z+i0))^{-1}W_2= \sum_{i,j=1}^{N_h} W_1Q_i(\eta)^* \phi\bigg(\frac{\Delta_g}{z}\bigg) (\Delta_g-(z+i0))^{-1}Q_j(\eta)W_2.
\]
The bound \eqref{eq_zreel_res} will follow if we show that for all $(i,j)$,  we have 
\begin{equation}
\label{eq_zreel_res_i_j}
\bigg\|W_1 Q_i(\eta)^* \phi\bigg(\frac{\Delta_g}{z}\bigg) (\Delta_g-(z+i0))^{-1} Q_j(\eta)W_2\bigg\|_{\mc{C}_{q}}\leq Cz^{-1+\frac{n}{2p}}\|W_1\|_{L^{2p}(M)} \|W_2\|_{L^{2p}(M)}.
\end{equation}
To that end, the pairs $(i,j)$ will be divided into three cases as in Proposition~\ref{prop:partition of unity}.   

In the first case, (2.i),  in view of \eqref{1stcase} and Corollary~\ref{cor:spec-localized-est}, we know that  the Schwartz kernel of  the operator $Q_i(\eta)^* \phi(\frac{\Delta_g}{z}) (\Delta_g-z-i0)^{-1}Q_j(\eta)$ is $\mc{O}(z^{-N})$ in $L^{2p'}(M\x M)$ with $1/p'+1/p=1$. Using this together with the fact that  $q\ge 2$ and H\"older's inequality, we get 
\begin{align*}
\bigg\|W_1Q_i(\eta)^*\phi\bigg(\frac{\Delta_g}{z}\bigg)& (\Delta_g-(z+i0))^{-1}Q_j(\eta)W_2\bigg\|_{\mc{C}_q}\\
&\le 
\bigg\|W_1Q_i(\eta)^*\phi\bigg(\frac{\Delta_g}{z}\bigg)(\Delta_g-(z+i0))^{-1}Q_j(\eta)W_2\bigg\|_{\mc{C}_2}\\
&\le \mathcal{O}(z^{-N})\|W_1\|_{L^{2p}(M)}\|W_2\|_{L^{2p}(M)},
\end{align*}
for any $N\in \N$, showing  \eqref{eq_zreel_res_i_j}.

In the second case, (2.ii), using Stone's formula, we write 
\begin{equation}
\label{eq_zreel_res_2_iii}
\begin{aligned}
W_1 Q_i(\eta)^* \phi\bigg(\frac{\Delta_g}{z}\bigg) (\Delta_g &- (z+i0))^{-1} Q_j(\eta) W_2 \\
&= W_1 Q_i(\eta)^* \phi\bigg(\frac{\Delta_g}{z}\bigg) (\Delta_g - (z-i0))^{-1} Q_j(\eta) W_2 \\ 
&+ \frac{\pi i}{\lambda} W_1 Q_i(\eta)^* dE_{\sqrt{\Delta_g}}(\lambda)  Q_j(\eta) W_2, \quad \lambda = \sqrt{z}.
\end{aligned}
\end{equation}
Then the estimate for the term involving the incoming resolvent in \eqref{eq_zreel_res_2_iii} follows exactly as in case (2.i).  On the other hand, 
 we have already proved the corresponding estimate  \eqref{spec-meas-bound-QiQj} for the spectral measure, which leads to the estimate \eqref{eq_zreel_res_i_j} in this case. 

In the third case, (2.iii), we get 
\begin{equation}
\label{eq_zreel_res_1}
\begin{aligned}
W_1 Q_i(\eta)^* \phi\bigg(\frac{\Delta_g}{z}\bigg) (\Delta_g - (z+i0))^{-1} Q_j(\eta) W_2 =W_1 Q_i(\eta)^* H_{-1, z,0}(\Delta_g) Q_j(\eta)W_2,
\end{aligned}
\end{equation}
where the operator $Q_i(\eta)^* H_{-1, z,0}(\Delta_g) Q_j(\eta)$ is defined in \eqref{eq_def_H_s_z_with_Q_eps=0}. The required estimate for this term  therefore is a consequence of Proposition~\ref{prop_cases_2_iii_3_iii}.  

In the low energy case, $0 < z \leq 1$,  the argument is similar. In cases (3.i) and (3.ii) we use Corollary~\ref{cor:spec-localized-est} together with Lemma~\ref{lem:K} and the bound \eqref{spec-meas-bound-QiQj} for the spectral measure to deduce the Schatten norm estimate. In case (3.iii), the argument is the same as for case (2.iii). 
This concludes the proof of the proposition. 
\end{proof}

\section{Bounds on individual eigenvalues. Proof of Theorem \ref{thm_indiviual_asymptotically conic}} 

\label{sec_4_individual}

In this section we shall follow some of the arguments of  \cite{Frank_2015} and \cite{Frank_Simon}, making some necessary changes due to the fact that we are no longer in the Euclidean setting.  

Let us recall that $n=\text{dim}(M)\ge 3$.  We have the following result which is a generalization of  \cite[Lemma 4.2]{Frank_2015}   to the case of the Laplace  operator on asymptotically conic manifolds.  
\begin{prop}
\label{prop_comp_Frank}
Let $V\in L^p(M)$ with $\frac{n}{2}\le p<\infty$.  The operator  $\sqrt{|V|}(\Delta_g+1)^{-\frac{1}{2}}$ is compact on $L^2(M)$. 
\end{prop}

\begin{proof}
We follow \cite[Lemma 4.2]{Frank_2015}. First  we shall show that 
\begin{equation}
\label{eq_indiv_prop_1}
\|W(\Delta_g+1)^{-\frac{1}{2}}\|_{\mathcal{L}(L^2(M), L^2(M))}\le C\| W\|_{L^{2p}(M)},\quad W\in L^{2p}(M).
\end{equation}
Indeed, we have 
\begin{equation}
\label{eq_indiv_prop_1_0_1}
(\Delta_g+1)^{-\frac{1}{2}}: L^2(M)\to H^{1}(M),
\end{equation}
is bounded, and therefore, by Sobolev's embedding $H^{1}(M)\subset L^{\frac{2n}{n-2}}(M)$, which is valid on an asymptotically conic manifold of dimension $n\ge 3$, see \cite[Proposition 2.1]{GuHa},  we get 
\begin{equation}
\label{eq_indiv_prop_1_0_2}
(\Delta_g+1)^{-\frac{1}{2}}: L^2(M)\to L^{\frac{2n}{n-2}}(M)
\end{equation}
 is also  bounded.  Using H\"older's inequality,  the logarithmic convexity of $L^p$ norms, and \eqref{eq_indiv_prop_1_0_1}, \eqref{eq_indiv_prop_1_0_2}, we obtain that 
\begin{align*}
\|W& (\Delta_g+1)^{-\frac{1}{2}} f\|_{L^2(M)}\le \|W\|_{L^{2p}(M)} \|(\Delta_g+1)^{-\frac{1}{2}} f\|_{L^{\frac{2p}{p -1}}(M)}\\
&\le \|W\|_{L^{2p}(M)} \|(\Delta_g+1)^{-\frac{1}{2}} f\|_{L^2(M)}^{1-\frac{n}{2p}}  \|(\Delta_g+1)^{-\frac{1}{2}} f\|_{L^{\frac{2n}{n-2}}(M)}^{\frac{n}{2p}}\\
&\le 
C  \|W\|_{L^{2p}(M)}\|f\|_{L^2(M)},
\end{align*}
showing \eqref{eq_indiv_prop_1}. 

Let $W_j\in C_0^\infty(M)$ be such that $W_j\to \sqrt{|V|}$ in $L^{2p}(M)$. By Rellich's compactness theorem, the operator $W_j(\Delta_g+1)^{-\frac{1}{2}}$ is compact on $L^2(M)$,  and it follows from \eqref{eq_indiv_prop_1} that 
$W_j(\Delta_g+1)^{-\frac{1}{2}}\to \sqrt{|V|}(\Delta_g+1)^{-\frac{1}{2}}$ in $\mathcal{L}(L^2(M), L^2(M))$. The proof is complete. 
\end{proof}

Setting
\[
\sqrt{V(x)}=\begin{cases}\frac{V(x)}{\sqrt{|V(x)|}},  & V(x)\ne 0,\\ 
0, & V(x)=0,
\end{cases}
\]
and combining Proposition \ref{prop_comp_Frank} with  \cite[Lemma B.1]{Frank_2015}, we get that the quadratic form 
\[
\| (\Delta_g)^{1/2}u\|^2_{L^2(M)}+(\sqrt{V}u, \sqrt{|V|}u)_{L^2(M)}, 
\]
equipped with the domain $H^{1}(M)$,  is closed and sectorial. Associated to the quadratic form is an $m$--sectorial operator 
with domain $\subset H^{1}(M)$,  which we shall denote by $\Delta_g+V$. The spectrum of $\Delta_g+V$ in $\C\setminus [0,\infty)$ consists of isolated eigenvalues of finite algebraic multiplicity, see \cite[Proposition B. 2]{Frank_2015}.

Now interpolating between the estimate, valid for $z\in \C\setminus[0,\infty)$, 
\[
\|(\Delta_g -z)^{-1}\|_{L^2(M)\to L^2(M)}=\frac{1}{d(z)},
\]
and the uniform estimate  \eqref{eq_KRS_manifold}, with $p=\frac{2(n+1)}{n+3}$, we obtain the following result. 
\begin{cor}
\label{cor_Guillarmou_Hassell_2014}
Let $(M,g)$ be an asymptotically conic non-trapping manifold of dimension $n\ge 3$. Then for all $p\in [\frac{2(n+1)}{n+3},2]$, there is a constant $C>0$ such that for all $z\in \C\setminus[0,\infty)$,  
\begin{equation}
\label{eq_KRS_manifold_non-uniform}
\| (\Delta_g -z)^{-1}\|_{L^{p}(M)\to L^{p'}(M)}\le Cd(z)^{(n+1)(\frac{1}{p}-\frac{1}{2})-1}|z|^{\frac{1}{2}-\frac{1}{p}}.
\end{equation}
\end{cor}

We shall now proceed to prove Theorem \ref{thm_indiviual_asymptotically conic}. In doing so we  shall follow \cite[Theorem 3.2]{Frank_Simon}.  Let $\lambda\in \C$ be an eigenvalue and $\psi\in H^{1}(M)$ be the corresponding eigenfunction of $\Delta_g+V$, 
\[
(\Delta_g+V)\psi=\lambda\psi.
\]

\textbf{(i)}  Let $0< \gamma\le \frac{1}{2}$.  Assume first that  $\lambda\in \C\setminus[0,\infty)$. Let us choose  $p>1$  such that  
\begin{equation}
\label{eq_8_1}
\gamma+\frac{n}{2}=\frac{p}{2-p}, 
\end{equation} 
and notice that then  $\frac{2n}{n+2}< p\le \frac{2(n+1)}{n+3}$ and $\frac{2(n+1)}{n-1}\le p'< \frac{2n}{n-2}$.  

By Sobolev's embedding, we have $\psi\in L^{\frac{2n}{n-2}}(M)$, and thus,  $\psi\in L^r(M)$ for $r\in [2,\frac{2n}{n-2}]$, by interpolation. In particular, $\psi\in L^{p'}(M)$, and by H\"older's inequality, we get 
\[
\|V\psi\|_{L^p(M)}\le \|V\|_{L^{\frac{p}{2-p}}(M)}\|\psi\|_{L^{p'}(M)}= \|V\|_{L^{\gamma+\frac{n}{2}}(M)}\|\psi\|_{L^{p'}(M)}.
\] 
We have
\[
\psi=(\Delta_g-\lambda)^{-1}(\Delta_g-\lambda)\psi=-(\Delta_g-\lambda)^{-1}(V\psi).
\]
Hence, using \eqref{eq_KRS_manifold}, we get 
\begin{equation}
\label{eq_8_1_1}
\begin{aligned}
\|\psi\|_{L^{p'}(M)}&\le \|(\Delta_g-\lambda)^{-1} \|_{L^p(M)\to L^{p'}(M)}\|V\psi\|_{L^p(\R^n)}\\
&\le C|\lambda|^{\frac{n}{2}(\frac{2}{p}-1)-1}\|V\|_{L^{\gamma+\frac{n}{2}}(M)}\|\psi\|_{L^{p'}(M)},
\end{aligned}
\end{equation}
which implies \eqref{eq_8_0} in view of 
\[
\frac{n}{2}\bigg(\frac{2}{p}-1\bigg)-1=-\frac{\gamma}{\gamma+\frac{n}{2}}. 
\]

Assume now that  $\lambda\in (0,\infty)$. Then 
for $\varepsilon>0$, we set 
\[
\psi_\varepsilon=(\Delta_g-\lambda-i\varepsilon)^{-1}(\Delta_g-\lambda)\psi=f_\varepsilon (\Delta_g)\psi,
\]
where 
\[
f_\varepsilon(t)=\frac{t-\lambda}{t-\lambda-i\varepsilon}, \quad t\in \R. 
\]
By the spectral theorem, we have
\[
\|\psi_\varepsilon-\psi\|_{L^2(M)}^2=\|f_\varepsilon (\Delta_g)\psi-\psi\|_{L^2(M)}^2=\int |f_\varepsilon(t)-1|^2d (E_{\Delta_g}(t)\psi, \psi)_{L^2(M)},
\]
where $d E_{\Delta_g}(t)$ is the spectral measure of $\Delta_g$. Using the dominated convergence theorem together with the fact that $f_\varepsilon(t)\to 1$ as $\varepsilon\to 0$ for all $t\ne \lambda$, and that $E_\lambda=0$ as $\lambda$ is not an eigenvalue of $\Delta_g$, we conclude that $\psi_\varepsilon\to \psi$ in $L^2(M)$.  

On the other hand, we have 
\[
\psi_\varepsilon=-(\Delta_g-\lambda-i\varepsilon)^{-1}(V\psi).
\]
Choosing $p>1$ satisfying \eqref{eq_8_1} and using \eqref{eq_KRS_manifold}, we obtain that 
\begin{equation}
\label{eq_8_2}
\|\psi_\varepsilon\|_{L^{p'}(M)}\le C|\lambda|^{\frac{n}{2}(\frac{2}{p}-1)-1}\|V\|_{L^{\gamma+\frac{n}{2}}(M)}\|\psi\|_{L^{p'}(M)},
\end{equation}
i.e. $\psi_\varepsilon$ is uniformly bounded in $L^{p'}(M)$. Passing to a subsequence, we may assume that there exists $\tilde \psi\in L^{p'}(M)$ such that $\psi_{\varepsilon}\to \tilde \psi$ in the weak $*$ topology of $L^{p'}(M)$.  It follows that  $ \psi=\tilde \psi\in L^{p'}(M)$.  By the lower semi-continuity of the norm and \eqref{eq_8_2}, we get 
\begin{equation}
\label{eq_8_2_1}
\|\psi\|_{L^{p'}(M)}\le \liminf_{\varepsilon\to 0}\|\psi_{\varepsilon}\|_{L^{p'}(M)}  \le  C|\lambda|^{\frac{n}{2}(\frac{2}{p}-1)-1}\|V\|_{L^{\gamma+\frac{n}{2}}(M)}\|\psi\|_{L^{p'}(M)},
\end{equation}
which shows \eqref{eq_8_0} when $\lambda\in (0,\infty)$.

\textbf{(ii)} Let $V\in L^{\frac{n}{2}}(M)$. Setting $p=\frac{2n}{n+2}$, and arguing as in the case (i) above, for $\lambda\in \C\setminus\{0\}$, we obtain that 
\[
\|\psi\|_{L^{p'}(M)}\le C \|V\|_{L^{\frac{n}{2}}(M)}\|\psi\|_{L^{p'}(M)}.
\]
The case $\lambda=0$ is handled similarly using that 
\[
\|(\Delta_g-i\varepsilon)^{-1}\|_{L^p(M)\to L^{p'}(M)}\le \mathcal{O}(1), 
\]
in view of \eqref{eq_KRS_manifold}. The claim (ii) follows. 

 \textbf{(iii)} Let $\gamma> \frac{1}{2}$, and let $\lambda\in \C\setminus[0,\infty)$ be an eigenvalue of $\Delta_g+V$, and $\psi\in H^{1}(M)$ be the corresponding eigenfunction.  Choosing  $p>1$ satisfying \eqref{eq_8_1}, we have
 $\frac{2(n+1)}{n+3}<p<2$ and  $2<p'<\frac{2(n+1)}{n-1}$. Using that $\psi\in L^{p'}(M)$ and \eqref{eq_KRS_manifold_non-uniform}, similarly to above, we obtain that 
 \begin{align*}
\|\psi\|_{L^{p'}(M)}&\le \|(\Delta_g-\lambda)^{-1} \|_{L^p(M)\to L^{p'}(M)}\|V\psi\|_{L^p(M)}\\
&\le C\delta(\lambda)^{(n+1)(\frac{1}{p}-\frac{1}{2})-1}  |\lambda|^{\frac{1}{2}-\frac{1}{p}}\|V\|_{L^{\gamma+\frac{n}{2}}(M)}\|\psi\|_{L^{p'}(M)}, 
\end{align*}
which implies \eqref{eq_8_0_iii}  in view of the fact that  $\frac{1}{p}=\frac{1+\gamma+\frac{n}{2}}{2(\gamma+\frac{n}{2})}$. This completes the proof of Theorem \ref{thm_indiviual_asymptotically conic}.

\section{Bounds on sums of eigenvalues for Schr\"odinger operators with complex potentials} 

\label{sec_bounds_sum_higher}

\subsection{Short range potentials. Proof of Theorem \ref{thm_main_sums_asymp}} 
Let $V\in L^p(M)$, $\frac{n}{2}\le p\le \frac{n+1}{2}$, and let $q=\frac{p(n-1)}{n-p}$. Then Theorem \ref{thm_resolvent_Schatten_laplacian} implies that for $z\in \C\setminus[0,\infty)$, we have $ \sqrt{V}(\Delta_g-z)^{-1}\sqrt{|V|}\in \mathcal{C}_{q}(L^2(M))$ and
\begin{equation}
\label{eq_100_1}
\| \sqrt{V}(\Delta_g-z)^{-1}\sqrt{|V|}\|_{\mathcal{C}_{q}(L^2(M))}\le C|z|^{-1+\frac{n}{2p}} \|V\|_{L^{p}(M)}. 
\end{equation}

We claim that the map 
\begin{equation}
\label{eq_100_10}
\C\setminus [0,\infty) \ni z\mapsto  \sqrt{V}(\Delta_g-z)^{-1}\sqrt{|V|}
\end{equation}
is holomorphic with values in $\mathcal{C}_{q}(L^2(M))$. First let us check that \eqref{eq_100_10} is holomorphic with values in 
$\mathcal{L}(L^2(M), L^2(M))$.  Indeed, letting $z_0\in \C\setminus [0,\infty)$, we write 
\begin{equation}
\label{eq_100_10_new_t}
\sqrt{V}(\Delta_g-z)^{-1} \sqrt{|V|} =\sqrt{V}\sum_{j=0}^\infty (z-z_0)^j(\Delta_g-z_0)^{-j-1}  \sqrt{|V|}
\end{equation}
and notice that 
\begin{align*}
\|\sqrt{V}(\Delta_g-z_0)^{-j-1}  \sqrt{|V|}\|_{\mathcal{L}(L^2(M), L^2(M))}\le \|\sqrt{V}(\Delta_g-z_0)^{-1}\|_{\mathcal{L}(L^2(M), L^2(M))}\\
 \|  (-\Delta-z_0)^{-1} \sqrt{|V|}\|_{\mathcal{L}(L^2(M), L^2(M))}\|(\Delta_g-z_0)^{-1}\|^{j-1}_{\mathcal{L}(L^2(M), L^2(M))}\le C^{j+1},
\end{align*}
for some $C>0$. Here we have used that the operators $\sqrt{V}(-\Delta-z_0)^{-1}$, $ (\Delta_g-z_0)^{-1} \sqrt{|V|}$ are bounded on $L^2(M)$ as seen by arguing as in the proof of \eqref{eq_indiv_prop_1}. This shows that the series \eqref{eq_100_10_new_t}
 converges in $\mathcal{L}(L^2(M), L^2(M))$ for $|z-z_0|$ small, and therefore, the map \eqref{eq_100_10} is holomorphic with values in $\mathcal{L}(L^2(M), L^2(M))$. In particular, if $T\in \mathcal{C}_1(L^2(M))$, i.e.  of trace class, the map 
 \begin{equation}
\label{eq_100_10_new_t_2}
\C\setminus [0,\infty) \ni z\mapsto  \langle \sqrt{V}(\Delta_g-z)^{-1}\sqrt{|V|}, T\rangle  
 \end{equation}
is holomorphic. Using the density of $\mathcal{C}_1(L^2(M))$ in $\mathcal{C}_{q'}(L^2(M))$, the bound \eqref{eq_100_1}, and  H\"older's inequality in Schatten classes, we conclude that the map \eqref{eq_100_10_new_t_2} is holomorphic for all $T\in \mathcal{C}_{q'}(L^2(M))$, establishing the claim.

Consider the holomorphic function 
\[
h(z):=\text{det}_{ \ceil{q}}(1+ \sqrt{V}(\Delta_g-z)^{-1}\sqrt{|V|}), \quad z\in \C\setminus [0,\infty),
\] 
where $ \ceil{q}$ is the smallest integer  $\ge q$, and $\text{det}_{ \ceil {q}}$ is the regularized determinant, see \cite[Chapter 9]{Simon_trace}. As explained in  \cite[proof of Theorem 16]{Frank_Sabin}, using \eqref{eq_100_1}, we get 
 \begin{equation}
\label{eq_11_1}
\log|h(z)|\le C\big\| \sqrt{V}(\Delta_g-z)^{-1}\sqrt{|V|}\big\|_{\mathcal{C}_{q}}^{q}\le C|z|^{(-1+\frac{n}{2p})q} \|V\|^{q}_{L^{p}(M)},
\end{equation}
uniformly in $z\in \C\setminus[0,\infty)$.

Combining Proposition \ref{prop_comp_Frank} and Lemma B.1 of \cite{Frank_2015}, we conclude that the following version of the Birman--Schwinger principle holds: $z\in \C\setminus [0,\infty)$ is an eigenvalue of $\Delta_g+V$ if and only if 
\begin{equation}
\label{eq_11_1_B_Schwinger}
\text{Ker}\,  (1+\sqrt{V}(\Delta_g-z)^{-1}\sqrt{|V|})\ne \{0\}. 
\end{equation}
An application of Lemma 3.2 of  \cite{Frank_2015} gives that \eqref {eq_11_1_B_Schwinger} is equivalent to the fact that $h(z)=0$ and that the order of vanishing of $h$ at $z$ agrees with the algebraic multiplicity of $z$ as an eigenvalue of  $\Delta_g+V$.  

At this point we are exactly in the same situation as in \cite[Theorem 16]{Frank_Sabin}. Here we may remark that the proof of Theorem 16 in  \cite{Frank_Sabin} is based on a result of Borichev, Golinskii and Kupin \cite{Borichev_Golinskii_Kupin}, concerning the distribution of zeros of a holomorphic function in the unit disc, growing rapidly at a boundary point.  The proof of Theorem  \ref{thm_main_sums_asymp} is therefore complete.

\subsection{Long range potentials. Proof of Theorem \ref{thm_sums_long_range_asym}} 

First we have the following result: let $\gamma\ge 1/2$. Then there exists a constant $C>0$ such that for all $W\in L^{2(\gamma+\frac{n}{2})}(M)$ and all $z\in\C\setminus[0,\infty)$,
\begin{equation}
\label{eq_long_Scatt}
\|W(\Delta_g-z)^{-1}W\|_{\mathcal{C}_{2(\gamma+\frac{n}{2})}}\le C d(z)^{-1+\frac{n+1}{2(\gamma+\frac{n}{2})}}|z|^{-\frac{1}{2(\gamma+\frac{n}{2})}}\|W\|^2_{L^{2(\gamma+\frac{n}{2})}(M)}.
\end{equation}
Indeed this follows as in \cite[Proposition 2.1]{Frank_2015} by interpolation between \eqref{toprove} with $p=\frac{n+1}{2}$ and the standard bound
\[
\|W(\Delta_g-z)^{-1}W\|_{L^2(M)\to L^2(M)}\le d(z)^{-1}\|W\|^2_{L^\infty(M)}.
\]
Now an application of \cite[Theorem 3.1]{Frank_2015} to the holomorphic family 
$K(z)= \sqrt{V}(\Delta_g-z)^{-1}\sqrt{|V|}$ completes the proof of Theorem \ref{thm_sums_long_range_asym} exactly in the same way as in   \cite[Theorem 1.2]{Frank_2015}.

\begin{appendix}

\section{Proof of Lemma \ref{Guillarmou_Hassell_remark}}
\label{app}

We shall follow the proof of Lemma 3.3 in \cite{GHS} closely. Let $a<b<c\le 0$ and let $\alpha:=a-c-1<-1$ and $\beta:=b-c-1<-1$.  We shall show the estimate \eqref{eq_10_4_conv} for   $\|(\lambda- i \varepsilon)^{b+it}* f\|_{L^\infty_\lambda}$, as the bound  \eqref{eq_10_4_conv} for $\|(\lambda+ i \varepsilon)^{b+it}* f\|_{L^\infty_\lambda}$ can be proved similarly. 

To that end, let $\chi_-^z$ be the family of distributions on $\R$ holomorphic in $z\in \C$ given by 
\[
\chi_-^z(\lambda)=\frac{\lambda_-^z}{\Gamma(z+1)}, \quad \Re\, z>-1,
\]
where 
\[
\lambda_-^z=\begin{cases} 0 & \text{if}\quad \lambda>0,\\
|\lambda|^z & \text{if}\quad \lambda< 0.
\end{cases}
\]
We have $\chi_-^z(-\lambda)=\chi_+^z(\lambda)$.  Recall from  \cite[Section 3.2)]{Hormander_books_1} that when $\Re\, z>-1$, we have
\begin{equation}
\label{(x-i0)^z} 
(\lambda-i0)^{z}=\lambda_+^z+e^{-i\pi z}\lambda_-^z,
\end{equation}
and from  \cite[Example 7.1.17]{Hormander_books_1} that for $\eps>0$ and $z\in\cc$, we have 
\begin{equation}
\label{fourierx-i0} 
\mc{F}((\lambda-i\eps)^{-z})(\xi)=2\pi e^{iz\pi/2}e^{\eps \xi}\chi_-^{z-1}(\xi),
\end{equation}
and 
\begin{equation}
\label{fourier_chi} 
\mc{F}(\chi_+^z)(\xi)= e^{-i(z+1)\pi/2}(\xi-i0)^{-z-1}.
\end{equation}

Consider the family of operators $A_t$ for $t\in\rr$ given by 
\begin{equation}
\label{eq_def_A_t}
A_t: C_0^\infty(\rr)\to \mc{D}'(\rr), \quad A_tf:=\eta_t* f, 
\end{equation}
where 
\begin{equation}
\label{eq_def_A_t_eta_t_c<0}
\hat{\eta}_t(\xi)= \frac{2\pi e^{i(-\beta-it)\pi/2-i\pi(c+1)}e^{\eps \xi}\xi_-^{-\beta-1-it}}{\Gamma(-b-it)(\sigma+e^{-i(\alpha+1)\pi/2}(\xi-i0)^{-\alpha-1})}, 
\end{equation}
when $c<0$, and 
\begin{equation}
\label{eq_def_A_t_eta_t_c=0}
\hat{\eta}_t(\xi)=\frac{2\pi e^{-i(b-1+it)\pi/2}e^{\eps \xi}\xi_-^{-b-it}}{\Gamma(-b-it)(\sigma-e^{-i\pi a/2}(\xi-i0)^{-a})},
 \end{equation}
 when $c=0$, 
and $\sigma\in \cc$, $|\sigma|=1$ and $\sigma\notin\{ ie^{-i\alpha\pi/2},-ie^{i\alpha\pi/2}, e^{ia\pi/2}\}$. 
In view of  \eqref{(x-i0)^z},  we see that  $\hat{\eta}_t\in \mc{S}'(\rr)$.

We notice that for all $t\in \R$,  $\hat{\eta_t}\in L^1_{\rm loc}(\rr)$. Furthermore, using that $|\frac{1}{\Gamma(-b-it)}|\le Ce^{\pi |t|}$,  we have,  for $|\xi|\ge 1$,
\begin{equation}
\label{eq_eta_l_1_-3}
|\p_{\xi}\hat{\eta_t}(\xi)|\le Ce^{\frac{3\pi|t|}{2} }(1+|t|)|\xi|^{-\beta+\alpha-1},
\end{equation}
and for $|\xi|\le 1$, we get
\begin{equation}
\label{eq_eta_l_1_-2}
|\p_{\xi}\hat{\eta_t}(\xi)|\le Ce^{\frac{3\pi|t|}{2}}(1+|t|)|\xi|^{-\beta-2},
\end{equation}
and therefore, 
\[ 
\pl_\xi \hat{\eta_t}\in L^p(\rr)\cap L^1(\rr,\cjg \xi\cjd^{\delta}d\xi) \textrm{ for  some } p\in(1,2), \, \delta>0.
\]
By Hausdorff--Young's inequality, we see that  $u(\lambda):=\lambda \eta_t(\lambda)\in L^{p'}(\rr)$ with $p'\in (2,\infty)$ being the dual exponent to $p$. We also have 
\begin{equation}
\label{eq_eta_l_1_-1}
\begin{aligned}
|u(\lambda)-u(\lambda')|&\le (2\pi)^{-1}\int |e^{i\xi \lambda}-e^{i\xi\lambda'}||\hat u(\xi)|d\xi\le C \int |\xi|^{\delta}|\lambda-\lambda'|^{\delta}|\hat u(\xi)|d\xi\\
&\le C|\lambda-\lambda'|^{\delta}\|\hat u\|_{L^1(\rr,\cjg \xi\cjd^{\delta}d\xi)},
\end{aligned}
\end{equation}
showing that $u=\lambda \eta_t\in C^{\delta}(\R)$. Thus,  by H\"older inequality, we get 
\begin{equation}
\label{eq_eta_l_1}
\int_\rr |\eta_t(\lambda)|d\lambda\leq C\Big(\int_{|\lambda|>1}|\lambda\eta_t|^{p'}d\lambda\Big)^{\frac{1}{p'}}+||\lambda\eta_t||_{C^{\delta}}
\int_{|\lambda|<1}|\lambda|^{-1+\delta}d\lambda<\infty.
\end{equation}
It follows from \eqref{eq_eta_l_1} combined with Hausdorff--Young's inequality, \eqref{eq_eta_l_1_-3},  \eqref{eq_eta_l_1_-2} and  \eqref{eq_eta_l_1_-1} that 
\[
 \|\eta_t\|_{L^1(\rr)}\leq C(1+|t|)e^{\frac{3\pi|t|}{2}},
\]
and therefore, $A_t$ extends as a bounded operator on $L^\infty$ with norm
\[ 
\|A_t\|_{L^\infty(\R)\to L^\infty(\R)}\leq C(1+|t|)e^{\frac{3\pi|t|}{2}},
\]
where the constant $C>0$ is independent of $\eps$ and $t$.

Next let $B$ be the operator 
\[
\begin{gathered} 
B: C_0^\infty(\rr)\to C^\infty(\rr), \quad Bf:= (\sigma\chi^{c}_++\chi_+^a)* f
\end{gathered}
\]
which is also equal to 
\begin{equation}
\label{eq_def_B_mult}
B=\mc{F}^{-1}\mu\mc{F}
\end{equation}
 with 
\begin{equation}
\label{eq_def_B_mult_1}
\mu(\xi) :=\sigma e^{-i(c+1)\pi/2}(\xi-i0)^{-c-1}+e^{-i(a+1)\pi/2}(\xi-i0)^{-a-1},
\end{equation}
in view of \eqref{fourier_chi}.    

If $c<0$ then $\mu\in L^1_{\rm loc}(\R)\cap C^{\infty}(\R\setminus\{0\})$. Using also the fact that the distribution $(\xi-i0)^z$ is of polynomial growth when $\Re\, z>-1$, we have $\mu\hat f\in L^1(\R)$ for any $f\in C^\infty_0(\R)$.  Thus, the operator 
$B:C_0^\infty(\R)\to L^\infty(\rr)$ is bounded.  

Now if $c=0$ then $Bf:= \sigma H*f+\chi_+^a* f$, where $H$ is the Heaviside function. The fact that the convolution with the  Heaviside function maps $C^\infty_0$ functions into $L^\infty$ functions implies that the operator $B:C_0^\infty(\R)\to L^\infty(\rr)$ is bounded also in the case $c=0$.

Thus, the composition $A_tB: C_0^\infty(\R)\to L^\infty(\R)$ is bounded in all cases $c\le 0$.  We claim that 
\begin{equation}
\label{eq_prod_A_tB}
A_tBf=(\lambda-i\eps)^{b+it}* f, \quad f\in C_0^\infty(\R).
\end{equation}
Indeed, \eqref{eq_prod_A_tB} follows from   \eqref{eq_def_A_t},      \eqref{eq_def_B_mult},   and the equality 
\[
\hat{\eta_t}\mu=\mc{F}((\lambda-i\eps)^{b+it})
\]
obtained from \eqref{eq_def_A_t_eta_t_c<0}, \eqref{eq_def_A_t_eta_t_c=0} \eqref{eq_def_B_mult_1}, and \eqref{fourierx-i0}. In the case $c=0$, we also use that 
\[
\xi_-^{-b-it}(\xi-i0)^{-1}=\xi^{-b-1-it}, \quad b<0.
\]

We thus get for all $\eps>0$ and $t\in\rr$
\begin{equation}
\label{eq_non_resc_1} 
\|(\lambda-i\eps)^{b+it}* f\|_{L^\infty} \leq C(1+|t|)e^{\frac{3\pi|t|}{2}}(\|\chi^{c}_+*f\|_{L^\infty}+\|\chi^{a}*f\|_{L^\infty}).
\end{equation}
Now a scaling argument as in the proof of Lemma 3.3 of \cite{GHS} finishes the proof. Indeed,  letting $f_\tau(\lambda)=f(\tau\lambda)$, we have 
\begin{equation}
\label{eq_non_resc_2} 
\chi^z_+*f_\tau(\lambda)=\tau^{-z-1}(\chi_+^z*f)(\tau\lambda), \quad 
(\lambda-i\eps)^z*f_\tau  (\lambda)=\tau^{-z-1}((\lambda-i\tau\eps)^z*f)(\tau\lambda)
\end{equation}
for all $\tau>0$ and $z\in\cc$. It follows from \eqref{eq_non_resc_1}  and \eqref{eq_non_resc_1} that for each $\tau>0$
\[
\tau^{-b}||(\lambda-i\tau\eps)^{b+it}* f||_{L^\infty} \leq C(1+|t|)e^{\frac{3\pi|t|}{2}}(\tau^{-c}||\chi^{c}_+*f||_{L^\infty}+\tau^{-a}||\chi_+^{a}*f||_{L^\infty})
\]
and choosing $\tau:=||\chi_+^{a}*f||^{1/(a-c)}_{L^\infty}||\chi_+^{c}*f||^{-1/(a-c)}_{L^\infty}$, we obtain the desired estimate \eqref{eq_10_4_conv}. The proof of Lemma \ref{Guillarmou_Hassell_remark} is complete.

\section{Microlocal structure of the spectrally localized resolvent}

\label{app_2}

In this appendix, we analyze the microlocal structure of the spectrally localized resolvent $\phi(\frac{\Delta_g}{z})(\Delta_g-(z\pm i0))^{-1}$,  where $z>0$ and  $\phi\in C^\infty_0(((1-\delta/4)^2,(1+\delta/4)^2))$ is such that $\phi(t)=1$ for $t\in ((1-\delta/8)^2,(1+\delta/8)^2)$,  for $\delta>0$  small. In doing so, we use the notation and results established in the works \cite{GHS}, \cite{GuHaSi_III}, and \cite{HaWu2008}. 

\begin{prop}
\label{prop_B_1}
Let $\phi$ be as above. For all $\mu > 0$, the operator $\phi(\frac{\Delta_g}{\mu^2})$ is a pseudodifferential operator 
in the following senses:

(i) \emph{High energy case}. For $h = \mu^{-1} \leq 2$, the operator $\phi(h^2 \Delta_g)$ is a semiclassical scattering pseudodifferential operator with microsupport in $\{ (z, \zeta) \mid |\zeta|_g \in ((1-\delta/4)^2,(1+\delta/4)^2) \}$ where $\zeta$ is the semiclassically-rescaled cotangent variable, i.e. $\zeta_i$ is the symbol of $-i h \partial_{z_i}$. 

(ii) \emph{Low energy case}. For $\mu \in (0, 2)$, the operator $\phi(\frac{\Delta_g}{\mu^2})$ is a pseudodifferential operator in the class $\Psi^0_k(M, \Omega_{k,b}^{1/2}) + \mathcal{A}^{\mathcal{E}}(M^2_{k,b}, \Omega_{k,b}^{1/2})$ where $\mathcal{E}$ is an index family for the boundary hypersurfaces of $M^2_{k,b}$, satisfying $\mathcal{E}_{\bfo} = 0$, $\mathcal{E}_{\zf} = n$, $\mathcal{E}_{\lbo} = \mathcal{E}_{\rbo} = n/2$, $\mathcal{E}_{\lb} = \mathcal{E}_{\rb} = \mathcal{E}_{\bfc} = \infty$. That is, it is the sum of a pseudodifferential operator in the class defined in \cite[Section 5]{GuHaSi_III} and a conormal function which is smooth across the diagonal, but has nontrivial behaviour at the boundary hypersurfaces $\lbo$ and $\rbo$. 
\end{prop}

\begin{proof}
(i) This follows by expressing the operator $\phi(h^2 \Delta_g)$ using the Helffer-Sj\"ostrand formula for the self-adjoint functional calculus, 
\[
\phi(h^2\Delta_g)=\frac{1}{2\pi i}\int_{\C} \bar{\p}\tilde \phi(z) (h^2\Delta_g-z)^{-1}d\overline{z}\wedge dz,
\]
where $\tilde \phi$ is an almost holomorphic extension of $\phi$, see \cite[Theorem 8.1]{DiSj}. In terms of the notation for the spaces of semiclassical scattering pseudodifferential operators used in  \cite{Vasy_Zworski_2000}, we have  $\phi(h^2\Delta_g)\in \Psi^{-\infty,0,0}_{\text{sc},h}(M)$. 

(ii) The same argument applies to show that the operator $\phi(\frac{\Delta_g}{\mu^2})$ is pseudodifferential in a neighbourhood of the diagonal on the space $M^2_{k,sc}$. We also need to understand the behaviour of the kernel of this operator away from the diagonal. Here, we recall from \cite{GuHaSi_III} that the spectral measure is conormal and vanishes to order $n-1$ at $\zf$, order $n/2 - 1$ at $\lbo$ and $\rbo$ and order $-1$ at $\bfo$ as a b-half-density on $\MMkb$, while it is Legendrian (oscillatory) at $\lb$, $\rb$ and $\bfc$. As a result, the integral 
\begin{equation}
\phi \big(\frac{\Delta_g}{\mu^2} \big) = \int \phi \big(\frac{\lambda^2}{\mu^2}\big) dE_{\sqrt{\Delta_g}}(\lambda)d\lambda
\label{phi-Delta}\end{equation}
is conormal on $\MMkb$ and vanishes to order $n$ at $\zf$, order $n/2$ at $\lbo$ and $\rbo$, order $0$ at $\bfo$ and to order $\infty$ at $\lb$, $\rb$ and $\bfc$. 
\end{proof}

\begin{rem} The pseudodifferential nature of $\phi(h^2\Delta_g)$ can also be proved via the spectral measure using the results of \cite{GuHaSi_III}. 
Recall from this article that the spectral measure $dE_{\sqrt{\Delta_g}}(\lambda)$ for $\lambda \geq 1$ is a Legendre distribution associated to a pair of  Legendre submanifolds $(L, L_2^\#)$, where $L$ is the flowout by (left) bicharacteristic flow starting from $N^* \Diagb \cap \Sigma_l$ where $N^* \Diagb$ is  the conormal bundle to the diagonal in $M^2_b$. Here $\Sigma_l$ denotes the `left' characteristic variety  of the operator $h^2 \Delta_g - 1$, that is, the set $\{ (z, \zeta, z', \zeta') \mid |\zeta|_g = 1 \}$ where the semiclassical symbol of $h^2 \Delta_g - 1$, acting in the left variable $z$, vanishes. Being a Legendre distribution, the spectral measure may be expressed (up to a trivial kernel, that is, one that is smooth and rapidly vanishing both as $h \to 0$ and as one approaches the boundary of $M^2_b$) as a finite sum of oscillatory integrals associated to neighbourhoods of the submanifold $L$. The phase function for this oscillatory integral takes the form $\lambda \Phi$, where $\Phi$ is independent of $\lambda$. If we then integrate in the $\lambda$ variable as in \eqref{phi-Delta} (with $h = \mu^{-1}$ in the high-energy case), then it is straightforward to check that the phase function $\lambda \Phi$ parametrizes the conormal bundle to the diagonal, and the result is a semiclassical scattering pseudodifferential operator of order $0$. 
\end{rem}

\begin{rem} It is not hard to see that the operator $\phi(\frac{\Delta_g}{\mu^2})$ is microlocally equal to the identity for $|\zeta|_g \in ((1-\delta/8)^2,(1+\delta/8)^2)$, where $\zeta$ is the rescaled cotangent variable. First, the operator $\phi(\frac{\Delta_g}{\mu^2})$ is elliptic in this region.  Next, choose a function $\phi_1$ supported in the interior of the region where $\phi = 1$. Then by functional calculus, $\phi_1(\frac{\Delta_g}{\mu^2}) = \phi(\frac{\Delta_g}{\mu^2})\phi_1(\frac{\Delta}{\mu^2})$, from which it follows that $\phi(\frac{\Delta_g}{\mu^2})$ is microlocally equal to the identity on the elliptic set of $\phi_1(\frac{\Delta_g}{\mu^2})$, which is an arbitrary subset of $\{ (z, \zeta) \mid |\zeta|_g\in ((1-\delta/8)^2,(1+\delta/8)^2)\}$.
\end{rem}

We next consider the microlocal structure of the spectrally localized resolvent. 

\begin{prop}
\label{prop_app_B_2}
The microlocal structure of the operator $\phi(\frac{\Delta_g}{z}) ( \Delta_g - (z \pm i0))^{-1}$, $z>0$, is as follows:

(i) High energy case. Here we use semiclassical notation and write $z = h^{-2}$. The operator $\phi(h^2 \Delta_g) (h^2 \Delta_g - (1 \pm i0))^{-1}$, acting on half-densities, lies in the same microlocal space as the semiclassical resolvent as detailed  in \cite[Theorem 1.1]{HaWu2008}, indeed in a `better'  space as the differential order is $-\infty$ rather than $-2$. 
That is, the spectrally localized resolvent is a sum of three terms $S_1 + S_2 + S_3$, where 
\begin{itemize}
\item $S_1$ is a semiclassical pseudodifferential operator of differential order $-\infty$ and semiclassical order $0$,

\item $S_2$ is an intersecting Legendre distribution associated to the conormal bundle $N^* \Diagb$ and to the propagating Legendrian $L$, and

\item $S_3$ is a conic Legendre pair associated to $L$ and to the outgoing Legendrian $L_2^\#$. 
\end{itemize}
Moreover, $S_2 + S_3$ are microlocally identical to the full resolvent in a neighbourhood of the characteristic variety $\Sigma_l$ of $h^2 \Delta_g - 1$. 

(ii) Low energy case. Let $z \in (0, 2)$. The operator $\phi(\frac{\Delta_g}{z}) (\Delta_g - (z \pm i0))^{-1}$, acting on half-densities, lies in the same microlocal space as the resolvent as detailed in \cite[Theorem 3.9]{GuHaSi_III}, indeed in a better space as the differential order is $-\infty$ rather than $-2$. In detail, the operator $\phi(\frac{\Delta_g}{z}) (\Delta_g - (z \pm i0))^{-1}$ can be decomposed as $S_1 + S_2 + S_3 + S_4$ (with $\sqrt{z}$ playing the role of the spectral parameter on $M^2_{k,b}$) where 
\begin{itemize}
\item $S_1 \in \Psi^{-\infty}(M, \Omega_{k,b}^{1/2})$ is a pseudodifferential operator of order $-\infty$ in the calculus of operators defined in \cite{GHS};

\item $S_2 \in I^{-1/2,\mcB}(\MMkb, (\Nscstar\Diagb, L^{\bfc}_+); \Omegakbh)$ is an intersecting Legendre distribution on $\MMkb$, microsupported close to $\Nscstar\Diagb$;

\item $S_3 \in I^{-1/2,(n-2)/2; (n-1)/2, (n-1)/2; \mcB}(\MMkb, (L^{\bfc}_+, \Lsharp_+); \Omegakbh)$ is a Legendre distribution on $\MMkb$ associated to the intersecting pair of Legendre submanifolds with conic points $(L^{\bfc}_+, \Lsharp_+)$, microsupported away from $\Nscstar\Diagb$;

\item $S_4$ is supported away from $\bfc$ and is such that $e^{\pm i \lambda r} e^{\pm i \lambda r'} R_4$ is
polyhomogeneous conormal  on $\MMkb$.

\end{itemize}

Here $\mcB = (\mcB_{\bfo}, \mcB_{\lbo}, \mcB_{\rbo}, \mcB_{\zf})$ is an index family with minimal exponents (i.e.\ order of vanishing) $\min \mcB_{\bfo} = -2$, $\min \mcB_{\lbo} = \min \mcB_{\rbo} = n/2 - 2$, $\min \mcB_{\zf} = 0$.  In addition $S_4$ vanishes to order $\infty$ at $\lb$ and $\bfc$ and to order $(n-1)/2$ at $\rb$. 

\end{prop}

\begin{cor} \label{cor:spec-localized-est}
The estimates \eqref{1stcase}, \eqref{2ndcase}, \eqref{1stcaselow} and \eqref{2ndcaselow} hold if the resolvent $(\Delta_g - (z \pm i0))^{-1}$ is replaced by the spectrally localized resolvent $\phi(\Delta_g/z) (\Delta_g - (z \pm i0))^{-1}$.
\end{cor}

\begin{proof}[Proof of Corollary~\ref{cor:spec-localized-est}] The proofs of these estimates only used the location of the wavefront set of the resolvent kernel, together with the vanishing orders of the resolvent on the boundary hypersurfaces of $M^2_{k,b}$ at $z = 0$. In view of Proposition~\ref{prop_app_B_2}, the same proof applies verbatim to the spectrally localized resolvent.
\end{proof}

\begin{proof}[Proof of Proposition~\ref{prop_app_B_2}]

(i) We study the composition of the operator $\phi(h^2 \Delta_g)$ with the incoming or outgoing resolvent, $(h^2 \Delta_g - (1 \pm i0))^{-1}$.  We know from \cite[Theorem 1.1]{HaWu2008} that the actual resolvent can be decomposed into a sum of three terms $R_1 + R_2 + R_3$
as in the proposition (except that $R_1$ will have differential order $-2$). 
We may assume that $R_2$ and $R_3$ are microsupported in the region  where $|\zeta|_g \in ((1-\delta/8)^2,(1+\delta/8)^2)$, and $R_1$ is  microsupported in the region where $|\zeta|_g \notin ((1-\delta/16)^2,(1+\delta/16)^2)$. The composition $S_1 := \phi(h^2 \Delta_g) R_1$ is another semiclassical pseudodifferential operator, of semiclassical order $0$ and differential order $-\infty$. On the other hand,  the operator $\phi(h^2 \Delta_g)$ is microlocally equal to the identity on the microsupport of $R_2$ and $R_3$, so using \cite[Section 7]{GHS}, we find that the composition of $\phi(h^2 \Delta_g)$ with $R_2 + R_3$ is equal to $R_2 + R_3$ up to an operator that is residual in all senses, that is, a smooth kernel that vanishes rapidly as $h \to 0$ or upon approach to the boundary of $M^2_b$. So we can take $S_2 = R_2$ and $S_3 = R_3$ up to a residual kernel. 

(ii) Similarly, in the low energy case the actual resolvent has a decomposition into $R_1 + R_2 + R_3 + R_4$ having properties as in the proposition (with $R_1$ of differential order $-2$). 
We also need to decompose the operator $\phi(\frac{\Delta_g}{z}) = B_1 + B_2$ into two parts, where $B_1$ is supported close to the diagonal on the space $M^2_{k,b}$, and $B_2$ has  empty wavefront set. This second piece $B_2$ can be taken to vanish to infinite order at $\bfc$, $\lb$ and $\rb$, and to be polyhomogeneous conormal to $\bfo, \lbo, \rbo$ and $\zf$ vanishing to order $0$ at $\bfo$, $n/2$ at $\lbo$ and $\rbo$ and order $n$ at $\zf$.  When we apply $B_1$ to the resolvent, the argument is just as in the high energy case, using \cite[Section 5]{GHS} instead of \cite[Section 7]{GHS}. 

To understand what happens when we apply $B_2$ to the resolvent, we view the composition of operators as pushforward of the product of the Schwartz kernels on a `triple space' $M^3_{k,b}$ down to $M^2_{k,b}$, as was done in the appendix of \cite{GuHa2008}. As a multiple of a nonvanishing b-half-density on $M^2_{k,b}$ we find that $B_2$ (multiplied by $|dk/k|^{1/2}$, $k=\sqrt{z}$, which is a purely formal factor) is polyhomogeneous conormal, with no log terms at leading order,  and vanishes to order $n$ at zf, $0$ at $\bfo$ and $n/2$ at $\lbo$ and $\rbo$. On the other hand, we can decompose the resolvent kernel as the sum of  $R_1 + R_2$, supported near the diagonal, and $R_3 + R_4$, which is microsupported in the set where  $|\zeta|_g \in ((1-\delta/8)^2,(1+\delta/8)^2)$, where $\zeta$ is the cotangent variable rescaled by a factor $\sqrt{z}$.

The composition of $B_2$ with $R_1+R_2$ can be treated by lifting both kernels to the space $M^3_{k,b}$ and pushing forward. Since $B_2$  has no wavefront set, the composition has no wavefront set, so it is polyhomogeneous conormal, and the order of vanishing can be read off as $n$ at $\zf$, $n/2$ at $\lbo$, $n/2 - 2$ at $\rbo$, $-2$ at $\bfo$, and $\infty$ at $\lb, \rb$ and $\bfc$. 
This lies in a better space than claimed in the proposition. 

The composition of $B_2$ with $R_3 + R_4$ can also be analyzed by lifting both kernels to $M^3_{k,b}$ and then pushing forward. Although $R_3 + R_4$ is not polyhomogeneous conormal at the boundary hypersurfaces $\bfc$, $\lb$ and $\rb$, when lifted to $M^3_{k,b}$ and multiplied by the lift of $B_2$, the rapid vanishing of $B_2$ at $\bfc$ and $\rb$ means that the product of the two kernels is rapidly decreasing as the `middle variable' (the right variable of $B_2$ and the left variable of $R_3 + R_4$) tends to the boundary. As for the right variable of $R_3 + R_4$, after multiplying the kernel of $R_3 + R_4$ by  $e^{\mp i\lambda r'}$ (where $r' = 1/x'$ is the right radial variable)  it becomes polyhomogeneous conormal also at $\rb$. So the product of the kernels $B_2$ (in the left and middle variables) and $(R_3 + R_4) e^{\mp i\lambda r'}$ (in the middle and right variables) on $M^3_{k,b}$ \emph{is} polyhomogeneous conormal. After pushing forward to $M^2_{k,b}$ a calculation similar to that done in \cite[Appendix]{GuHa2008} shows that the result is $e^{\mp i\lambda r'}$ times a polyhomogeneous kernel which vanishes to order $n-2$ at $\zf$, $-2$ at $\bfo$, $\min(n/2, n-2)$ at $\lbo$, $n/2 - 2$ at $\rbo$, $(n-1)/2$ at $\rb$ and $\infty$ at $\lb$ and $\bfc$, with no log terms to leading order except possibly at $\lbo$ in the case $n=4$. Again this is in a better space than is claimed in the proposition. This completes the proof. 
  
\end{proof}
  
\end{appendix}

\section*{Acknowledgements}

We are grateful to the referees for their very helpful comments and suggestions, 
and we thank Julien Sabin and Adam Sikora for useful discussions. 
This project has received funding from the European Research Council (ERC) under the European Union’s Horizon 2020 research and innovation programme (grant agreement No. 725967)
The research of C.G. was also partially supported by ANR grants 13-BS01-0007-01.  
A.H.  acknowledges the support of the Australian Research Council through Discovery Grants DP150102419 and DP160100941. 
The research of K.K. is partially supported by the National Science Foundation(DMS 1500703, DMS 1815922). C.G. finally thanks the hospitality of the Mathematical Sciences Institute at ANU where part of this work was done, and ARC grant DP160100941 for supporting the visit.

\end{document}